\documentclass[leqno]{article}

\usepackage[frenchb,english]{babel}
\usepackage[utf8]{inputenc}
\usepackage{amsmath}
\usepackage{amssymb}
\usepackage{amsfonts}
\usepackage{enumerate}
\usepackage{vmargin}
\usepackage[all]{xy}
\usepackage{mathrsfs}
\usepackage{mathtools}
\usepackage{lmodern}
\usepackage[colorlinks=true,linkcolor=blue,pagebackref=true]{hyperref}
\usepackage{comment}
\setmarginsrb{3cm}{3cm}{3.5cm}{3cm}{0cm}{0cm}{1.5cm}{3cm}

\let\O\relax 
\newcommand{\O}{\ensuremath{\mathbb{O}}}

\newcommand{\B}{\mathrm{B}} 
\newcommand{\I}{\mathrm{I}} 
\newcommand{\III}{\mathrm{III}}

\newcommand{\K}{\ensuremath{\mathbb{K}}}
\let\L\relax 
\newcommand{\L}{\mathrm{L}}

\newcommand{\M}{\mathrm{M}}
\newcommand{\CB}{\mathrm{CB}}

\newcommand{\R}{\ensuremath{\mathbb{R}}}
\newcommand{\C}{\ensuremath{\mathbb{C}}}

\newcommand{\W}{\mathrm{W}}

\newcommand{\Id}{\mathrm{Id}}
\newcommand{\TRO}{\mathrm{TRO}}

\newcommand{\la}{\langle}\newcommand{\ra}{\rangle}
\renewcommand{\leq}{\ensuremath{\leqslant}}
\renewcommand{\geq}{\ensuremath{\geqslant}}

\newcommand{\qed}{\hfill \vrule height6pt  width6pt depth0pt}
\newcommand{\norm}[1]{ \| #1  \|}
\newcommand{\bnorm}[1]{ \big\| #1  \big\|}

\newcommand{\bgnorm}[1]{ \bigg\| #1  \bigg\|}

\newcommand{\xra}{\xrightarrow}
\newcommand{\co}{\colon}

\newcommand{\ot}{\otimes}
\newcommand{\ovl}{\overline}

\newcommand{\D}{\mathrm{D}}
\newcommand{\dec}{\mathrm{dec}}
\newcommand{\pdec}{\mathrm{pdec}}
\newcommand{\reg}{\mathrm{reg}}
\newcommand{\cb}{\mathrm{cb}}
\newcommand{\Dec}{\mathrm{Dec}}
\newcommand{\PDec}{\mathrm{PDec}}

\let\i\relax 
\newcommand{\i}{\mathrm{i}}
\newcommand{\ov}{\overset}
\newcommand{\Asym}{\mathrm{Asym}} 
\newcommand{\sa}{\mathrm{sa}}
\newcommand{\JC}{\mathrm{JC}}
\newcommand{\JW}{\mathrm{JW}}
\newcommand{\JB}{\mathrm{JB}}
\newcommand{\JBW}{\mathrm{JBW}}
\newcommand{\w}{\mathrm{w}} 

\renewcommand{\d}{\mathop{}\mathopen{}\mathrm{d}} 
\newcommand{\e}{\mathrm{e}} 
\renewcommand{\d}{\mathop{}\mathopen{}\mathrm{d}} 
\let\cal\relax
\newcommand{\cal}{\mathcal}

\DeclareMathOperator{\tr}{Tr} 
\DeclareMathOperator{\Ran}{Ran} 
\let\Re\relax 
\DeclareMathOperator{\Re}{Re} 
\newtheorem{thm}{Theorem}[section]
\newtheorem{defi}[thm]{Definition}
\newtheorem{prop}[thm]{Proposition}
\newtheorem{conj}[thm]{Conjecture}
\newtheorem{quest}[thm]{Question}

\newtheorem{cor}[thm]{Corollary}
\newtheorem{lemma}[thm]{Lemma}

\newtheorem{remark}[thm]{Remark}
\newtheorem{prob}[thm]{Problem}
\newtheorem{example}[thm]{Example}

\newenvironment{proof}[1][]{\noindent {\it Proof #1} : }{\hbox{~}\qed
\smallskip
}

\usepackage{tocloft}
\numberwithin{equation}{section}
\usepackage[nottoc,notlot,notlof]{tocbibind}

\let\OLDthebibliography\thebibliography
\renewcommand\thebibliography[1]{
  \OLDthebibliography{#1}
  \setlength{\parskip}{0pt}
  \setlength{\itemsep}{0pt plus 0.3ex}
}

\begin{document}
\selectlanguage{english}
\title{\bfseries{Contractively decomposable projections on noncommutative $\L^p$-spaces}}
\date{}
\author{\bfseries{C\'edric Arhancet}\thanks{Albi, FRANCE, cedric.arhancet@protonmail.com, \href{http://sites.google.com/site/cedricarhancet}{https://sites.google.com/site/cedricarhancet}}}

\maketitle

\begin{abstract}
We describe and characterize the contractively decomposable projections on noncommutative $\L^p$-spaces. Our result relies on a new lifting result for decomposable maps of independent interest and on some tools from ergodic theory. Our theorem is new even for finite-dimensional Schatten spaces. Our description allows us to connect this topic with $\W^*$-ternary rings of operators and a slight generalization of our result for more general projections makes $\JBW^*$-triples appear in this context. We also prove that all rectangular $\L^p$-spaces associated with $\W^*$-ternary rings of operators arise as contractively decomposable complemented subspaces of noncommutative $\L^p$-spaces. Finally, we introduce a notion of $\L^p$-space associated to each $\sigma$-finite $\JBW^*$-triple and we explain the link with the context of this paper. 
\end{abstract}


\makeatletter
 \renewcommand{\@makefntext}[1]{#1}
 \makeatother

\tableofcontents

\section{Introduction}
\label{Introduction}

The investigation of the structure of projections and contractively complemented subspaces is a well-established topic of Banach space geometry. Recall that a bounded operator $P \co X \to X$ on a Banach space $X$ is a projection if $P^2=P$. If a subspace $Y$ of $X$ is the range of a contractive linear projection, we say that it is contractively complemented. See the surveys \cite{Ran01} and \cite{Mos06} for more information on the literature on the matter. Observe that by \cite[Theorem 6]{CoS70}, a subspace of a smooth Banach space $X$ can be the range of at most one projection of norm one. 

Ando proved in \cite{And66} (see also \cite{Dou65} for the case $p=1$) that the contractively complemented subspaces of a classical (=commutative) $\L^p$-space $\L^p(\Omega)$ for a finite measure space $\Omega$ are all isometrically isomorphic to an $\L^p$-space where $1 \leq p <\infty$. Moreover, he described the structure of these projections with the help of conditional expectations and multiplication operators. We refer to the papers \cite{Tza69}, \cite{BeL74}, \cite[Theorem 3 p.~162]{Lac74} and references therein for the case of general measures. 


On noncommutative $\L^p$-spaces, we cannot expect anything so simple. Indeed, if $\sigma \co S^p \to S^p$ denotes the transpose map on the Schatten space $S^p\ov{\mathrm{def}}{=} \{x \co \ell^2 \to \ell^2 : \norm{x}_p \ov{\mathrm{def}}{=} \bigl(\tr(\vert x\vert^p)\bigr)^{\frac{1}{p}}<\infty\}$, then the map $P \ov{\mathrm{def}}{=}\frac{1}{2}(\Id_{S^p} +\sigma) \co S^p \to S^p$ is a contractive projection on the subspace of symmetric matrices of $S^p$, in sharp contrast with the classical setting of measure spaces.

However, Arazy and Friedman succeeded in \cite{ArF78} and \cite{ArF92} in describing the structure of contractively complemented subspaces of $S^p$ for any $1 \leq p \leq \infty$. Such a subspace is isometrically isomorphic to a $\ell^p$-sum of $S^p$-Cartan factors of type I-IV. Recall that these Cartan factors are rectangular spaces of operators, spaces of antisymmetric operators, spaces of symmetric operators and complex spin factors. Moreover, the authors \cite[p.~99]{ArF92} explicitly raise the problem of describing contractively complemented subspaces of general noncommutative $\L^p$-spaces. 

Actually, Friedman and Russo showed in \cite{FrB85} that the range of a contractive projection on a noncommutative $\L^1$-space (=predual of von Neumann algebra) is isometric to the predual of a $\JW^*$-triple. Recall that a $\JW^*$-triple is a weak* closed subspace of the space $\B(H,K)$ of bounded operators between Hilbert spaces $H$ and $K$ which is closed under the triple product $(x,y,z) \mapsto xy^*z+zy^*x$. Furthermore, it is proved in \cite{NO02} that \textit{completely} contractively complemented subspaces of noncommutative $\L^1$-spaces coincide isometrically with the preduals of $\W^*$-ternary rings of operators. This concept has its roots in the paper \cite{Hes62} of Hestenes and has been studied by many authors. Recall that a $\W^*$-ternary ring of operators is a weak* closed subspace of the space $\B(H,K)$ for some Hilbert spaces $H$ and $K$ which is also closed under the triple product $(x,y,z) \mapsto xy^*z$. Ruan introduced a type decomposition of these objects in \cite{Rua04}. Finally, note that $\W^*$-ternary rings of operators, like $\mathrm{C}^*$-algebras, carry a natural operator space structure and injective dual operator spaces coincide with the $\W^*$-ternary rings of operators, see \cite{EOR01}. We refer to \cite{BFT12}, \cite{BuT19}, \cite{BuT13a}, \cite{BuT13b}, \cite{DoR07}, \cite{Ham92}, \cite{Ham99}, \cite {KaR02}, \cite{NeR03}, \cite{PlR19}, \cite{SaS13}, \cite{SaS17} and \cite{Zet83} for other papers on this topic.

Finally, in \cite[Theorem 1.1]{LRR09}, it is shown that any \textit{completely} 1-complemented subspace of $S^p$ is isometric to a direct sum of spaces of the form $S^p(H,K)$, where $H$ and $K$ are Hilbert spaces for $1 < p < \infty$. The proof relies on the explicit description of contractively complemented subspaces of $S^p$ of Arazy and Friedman.

The structure of 2-positive contractive projections on arbitrary noncommutative $\L^p$-spaces was completely elucidated in the paper \cite{ArR19}, three decades after the publication of \cite{ArF92}. The range of such a projection is completely order and completely isometrically isomorphic to some noncommutative $\L^p$-space. Furthermore, a description of 2-positive contractive projections was provided which was new even for Schatten spaces. The approach relies on a symmetric two-sided change of density and a lifting argument of the projection at the level $p=\infty$. 

In \cite{Arh20}, we investigated more generally the case of positive contractive projections. We highlighted the role of Jordan conditional expectations in this topic and we showed in a large number of cases that the range of a positive contractive projection is isometric to a \textit{nonassociative} $\L^p$-space associated to a $\JW^*$-algebra, a notion that we defined in \cite{Arh23}. A $\JW^*$-algebra can be viewed as a Jordan generalization of von Neumann algebras and this notion was introduced by Edwards \cite{Edw80} (see also \cite{You78}).

The goal of this paper is to investigate the case of contractively decomposable projections. Decomposable maps are a generalization of completely positive maps. Recall that a linear map $T \co \L^p(\cal{M}) \to \L^p(\cal{N})$ between noncommutative $\L^p$-spaces associated to von Neumann algebras $\cal{M}$ and $\cal{N}$ is decomposable \cite[(3.2)]{JuR04} if there exist bounded linear maps $v_1,v_2 \co \L^p(\cal{M}) \to \L^p(\cal{N})$ such that the linear map
\begin{equation}
\label{Matrice-2-2-Phi-intro}
\Phi
\ov{\mathrm{def}}{=}\begin{bmatrix}
   v_1  &  T \\
   T^\circ  &  v_2  \\
\end{bmatrix}
\co S^p_2(\L^p(\cal{M})) \to S^p_2(\L^p(\cal{N})), \quad \begin{bmatrix}
   a  &  b \\
   c &  d  \\
\end{bmatrix}\mapsto 
\begin{bmatrix}
   v_1(a)  &  T(b) \\
   T^\circ(c)  &  v_2(d)  \\
\end{bmatrix}
\end{equation}
is completely positive, where $T^\circ(c) \ov{\mathrm{def}}{=} T(c^*)^*$.  In this case, we let
\begin{equation}
\label{Norm-dec-intro}
\norm{T}_{\dec,\L^p(\cal{M}) \to \L^p(\cal{N})}
\ov{\mathrm{def}}{=} \inf \max\{\norm{v_1},\norm{v_2}\}
\end{equation}
where the infimum is taken over all maps $v_1$ and $v_2$. We say that $T$ is contractively decomposable if $\norm{T}_{\dec,\L^p(\cal{M}) \to \L^p(\cal{N})} \leq 1$. Note that if the von Neumann algebras $\cal{M}$ and $\cal{N}$ are hyperfinite, it is equivalent to saying that $T$ is contractively regular by \cite[Theorem 3.24]{ArK23}, which means that for any operator space $E$, the map $T \ot \Id_E$ induces a contraction between the vector-valued noncommutative $\L^p$-spaces $\L^p(\cal{M},E)$ and $\L^p(\cal{N},E)$. 

Our main result is the following theorem which describes precisely the structure of contractively decomposable projections acting on noncommutative $\L^p$-spaces. Roughly speaking, it says that such a projection is induced at the level $p=\infty$ by a contractively decomposable projection modulo a suitable nonsymmetric two-sided change of density.

\begin{thm}
\label{main-th-ds-intro}
Let $\cal{M}$ be a $\sigma$-finite von Neumann algebra equipped with a normal faithful state $\varphi$. Suppose $1 < p<\infty$. A bounded map $P \co \L^p(\cal{M}) \to \L^p(\cal{M})$ is a contractively decomposable projection if and only if there exist there exist a normal faithful positive linear form $\phi$ on the von Neumann algebra $\M_2(\cal{M})$, two positive elements $h,k \in \L^p(\cal{M})$ with support projection $s(h)$ and $s(k)$, some linear maps $w \co \cal{M} \to s(h)\cal{M}s(k)$, $u_1 \co \cal{M} \to s(h)\cal{M}s(h)$, $u_2 \co \cal{M} \to s(k)\cal{M}s(k)$ such that the map $Q \ov{\mathrm{def}}{=} \begin{bmatrix}
u_1 & w \\
w^\circ  & u_2 \\
\end{bmatrix}$ and the element $H \ov{\mathrm{def}}{=} \begin{bmatrix}
h & 0 \\
0  & k \\
\end{bmatrix}$ of $S^p_2(\L^p(\cal{M}))$ satisfy the following properties:
\begin{enumerate}
\item for any $x \in \L^p(\cal{M})$ we have $P(x)=P(s(h)xs(k))$,

\item $Q$ is a weak* continuous unital contractive completely positive map and its restriction on the von Neumann algebra $s(H)\M_2(\cal{M})s(H)$ is a faithful projection,

\item $s(H)$ belongs to the centralizer of $\phi$ and $\phi_{s(H)}=(\tr \ot \varphi)_{s(H)}$,

\item for any $y \in s(H)\M_2(\cal{M})s(H)$ we have
$
(\tr \ot \tr_{\varphi})(H^p Q(y))
=(\tr \ot \tr_{\varphi})(H^p y)$,


\item for any $x \in \cal{M}$, we have 
\begin{equation}
\label{relevement-dec-intro}
\quad P\big(h^{\frac{1}{2}}xk^{\frac{1}{2}}\big)
= h^{\frac{1}{2}} w(x) k^{\frac{1}{2}}, \qquad x \in \cal{M}.
\end{equation}
\end{enumerate}
In this case, the restriction of $w$ on $s(h)\cal{M}s(k)$ is a weak* continuous contractively decomposable projection and the range of this map admits a structure of a $\W^*$-ternary ring of operators.
\end{thm}
 
In the spirit of the paper \cite{Arh20}, it is expected than the range of a contractively decomposable projection is some kind of $\L^p$-space of a $\W^*$-ternary ring of operators. We did not clarify this point in this paper. In the opposite direction, we observe in Section \ref{Sec-complementation} that all rectangular $\L^p$-spaces associated with $\W^*$-ternary rings of operators arise as contractively decomposable complemented subspaces of noncommutative $\L^p$-spaces.

Actually, our main result admits a generalization (see Theorem \ref{main-th-ds-bis}) to a more general class of contractive projections, the contractively $n$-pseudo-decomposable projections defined in Section \ref{Sec-pseudo}. In this case, the range of $w$ admits a structure of a $\JW^*$-triple. Note that $\JW^*$-triples are particular cases of $\JBW^*$-triples. It is well-known that the latter objects are strongly connected to bounded symmetric domains in complex Banach spaces \cite{Isi19} \cite{Kau83} and to physical systems \cite{Fri05}.

Finally, we construct $\L^p$-spaces associated to suitable $\JBW^*$-triples in Section \ref{Sec-Lp-JBW} by using complex interpolation generalizing the construction of nonassociative $\L^p$-spaces associated to $\JBW^*$-algebras of our previous paper \cite{Arh23}. Moreover, we conjecture in Conjecture \ref{conj1} that a contractively complemented subspace of a noncommutative $\L^p$-space is isometric to the $\L^p$-space of a $\JW^*$-triple.






\paragraph{Approach of the paper} Suppose $1<p<\infty$. Consider a contractively decomposable projection $P \co \L^p(\cal{M}) \to \L^p(\cal{M})$. The first step of our approach consists in showing that we can suppose in \eqref{Matrice-2-2-Phi-intro} with $T=P$ that the maps $v_1$ and $v_2$ are (completely positive) contractive \textit{projections} $P_1$ and $P_2$ using ergodic theory (see Proposition \ref{prop-cd-proj-v1v2-proj}). As a second step, we will show that a non-symmetric two-sided change of density allows to reduce the analysis to the case $p=\infty$ by some lifting argument (see Theorem \ref{Lifting-decomposable}) relying on the lifting result of \cite{Arh20} achieved by a symmetric two-sided change of density (see Theorem \ref{Th-relevement-cp}). This lifting trick is of independent interest. Of course, a difficulty is the choice of suitable densities in order to perform the change of density. 
It should also be noted that the use of \textit{non-tracial} Haagerup noncommutative $\L^p$-spaces is essential in general for the case of \textit{tracial} noncommutative $\L^p$-spaces. 

\paragraph{Structure of the paper}
The paper is organized as follows. Section \ref{Haagerup-noncommutative} gives a brief presentation of Haagerup noncommutative $\L^p$-spaces, followed by some preliminary facts that are at the root of our results. Section \ref{Sec-prelim-TRO} contains background on TRO's, Jordan algebras, $\JBW^*$-triples which are necessary for our paper. In Section \ref{Sec-pseudo}, we investigate properties of $n$-pseudo-decomposable maps which are a slight generalization of decomposable maps. The proofs in this part are (essentially) similar to the proofs of \cite{ArK23}. In Section \ref{Sec-complementation}, we show that all rectangular $\L^p$-spaces associated with $\W^*$-ternary rings of operators arise as contractively decomposable complemented subspaces of noncommutative $\L^p$-spaces.  In Section \ref{Sec-lifting}, we present our new lifting result for ($n$-pseudo-)decomposable maps on noncommutative $\L^p$-spaces. Section \ref{Sec-Contractively-decomposable-projections} contains a proof of (our main result) Theorem \ref{main-th-ds-intro} which relies on our lifting result and ergodic theory. Indeed, we provide a more general statement for $n$-pseudo-decomposable maps which are defined in Section \ref{Sec-pseudo}. In Section \ref{Sec-Lp-JBW}, we introduce $\L^p$-spaces associated to $\sigma$-finite $\JBW^*$-triples with the help of complex interpolation. Our construction generalizes the one of nonassociative $\L^p$-spaces associated to $\sigma$-finite $\JBW^*$-algebras introduced in \cite{Arh23}. Finally, in Section \ref{Sec-questions}, we describe open questions raised by the contents of this paper.

\section{Preliminaries on Haagerup noncommutative $\L^p$-spaces and Banach spaces}
\label{Haagerup-noncommutative}

The readers are referred to the books \cite{EfR00}, \cite{Pau02}, \cite{Pis03} and \cite{Pis98} for details on operator spaces and completely bounded maps and to \cite{Kos14}, \cite{PiX03}, \cite{Ray03} and \cite{Ter81} for information on  noncommutative $\L^p$-spaces and references therein.

It is well-known that there are several equivalent constructions of noncommutative $\L^p$-spaces associated with a von Neumann algebra. In Section \ref{Sec-lifting} and Section \ref{Sec-Contractively-decomposable-projections}, we will use Haagerup noncommutative $\L^p$-spaces introduced in \cite{Haa79} and described more precisely in \cite{Ter81}. We denote by $s(x)$ the support of a positive operator $x$. If $\cal{M}$ is a von Neumann algebra equipped with a normal semifinite faithful trace, then the topological $*$-algebra of all (unbounded) $\tau$-measurable operators $x$ affiliated with $\cal{M}$ is denoted by $\L^0(\cal{M},\tau)$. 



In the sequel, we fix a normal semifinite faithful weight $\varphi$ on a von Neumann algebra $\cal{M}$ acting on a Hilbert space $H$. The one-parameter modular automorphisms group associated with $\varphi$ is denoted by $\sigma^\varphi=(\sigma_t^\varphi)_{t \in \R}$ \cite[p.~92]{Tak03}. 


For $1 \leq p <\infty$, the spaces $\L^p(\cal{M})$ are constructed as spaces of measurable operators with respect to some semifinite bigger von Neumann algebra, namely, the crossed product $\tilde{\cal{M}} \ov{\mathrm{def}}{=} \cal{M} \rtimes_{\sigma^\varphi} \R$ of $\cal{M}$ by the modular automorphisms group $\sigma^\varphi$, that is, the von Neumann subalgebra of $\B(\L^2(\R,H))$ generated by the operators $\pi(x)$ and $\lambda_s \ot \Id_H$, where $x \in \cal{M}$ and $s \in \R$, defined by
\begin{equation}
\label{Translations}
\big(\pi(x)\xi\big)(t) \ov{\mathrm{def}}{=} \sigma^{\varphi}_{-t}(x)(\xi(t))
\quad \text{and} \quad
(\lambda_s \ot \Id_H)(\xi(t)) \ov{\mathrm{def}}{=} \xi(t-s), \quad t \in \R, \ \xi \in \L^2(\R,H).
\end{equation} 
For any $s \in \R$, let $W(s)$ be the unitary operator on $\L^2(\R,H)$ defined by
\begin{equation}
\label{Def-W}
\big(W(s)\xi\big)(t)
\ov{\mathrm{def}}{=} \e^{-\i s t} \xi(t),\quad \xi \in \L^2(\R,H).
\end{equation}
The dual action $\widehat{\sigma} \co \R \to \B(\tilde{\cal{M}})$ on $\cal{M}$ \cite[p.~260]{Tak03} is given by
\begin{equation}
\label{Dual-action}
\widehat{\sigma}_s(x)
\ov{\mathrm{def}}{=} W(s)xW(s)^*,\quad x \in \tilde{\cal{M}},\ s \in \R.
\end{equation}
Then, by \cite[Lemma 3.6]{Haa78a} or \cite[p.~259]{Tak03}, $\pi(\cal{M})$ is the fixed subalgebra of $\tilde{\cal{M}}$ under the family of automorphisms $\widehat{\sigma}_s$:
\begin{equation}
\label{carac-Pi-de-M}
\pi(\cal{M})
=\big\{x \in \tilde{\cal{M}}\ : \ \widehat{\sigma}_s(x)=x \quad \text{for all }s \in \R\big\}.
\end{equation}

We identify $\cal{M}$ with the subalgebra $\pi(\cal{M})$ in the crossed product $\tilde{\cal{M}}$. If $\psi$ is a normal semifinite weight on $\cal{M}$, we denote by $\widehat{\psi}$ its Takesaki's dual weight on the crossed product $\cal{M}$. We can give the following definition of \cite{Haa78b} using the theory of operator valued weights. 
Indeed, Haagerup introduces an operator valued weight $T \co \cal{M}^+ \to \bar{\cal{M}}^+$ with values in the extended positive part\footnote{\thefootnote. If $\cal{M}=\L^\infty(\Omega)$ then $\bar{\cal{M}}^+$ identifies to the set of equivalence classes of measurable functions $\Omega \to [0,\infty]$.} $\bar{\cal{M}}^+$ of $\cal{M}$, formally defined by 
\begin{equation}
\label{Operator-valued}
T(x)
=\int_\R \widehat{\sigma}_s(x)\d s
\end{equation}
and shows that for a normal semifinite weight $\psi$ on $\cal{M}$, its dual weight is 
\begin{equation}
\label{Def-poids-dual}
\widehat{\psi} \ov{\mathrm{def}}{=} \bar \psi\circ T
\end{equation}
where $\bar\psi$ denotes the natural extension of the normal weight $\psi$ to $\bar{\cal{M}}^+$. This dual weight  satisfies the $\widehat{\sigma}$-invariance relation $\widehat{\psi} \circ \widehat{\sigma}=\widehat{\psi}$, see \cite[(10)]{Ter81}.

By \cite[p.~301]{Str81} \cite[Theorem 3.7]{Haa78a} \cite[Chap.~II, Lemma 1]{Ter81}, the map $\psi \to \widehat{\psi}$ is a bijection from the set of normal semifinite weights on $\cal{M}$ onto the set of normal semifinite $\widehat{\sigma}$-invariant weights on $\tilde{\cal{M}}$.

Recall that by \cite[Lemma 5.2 and Remark p.~343]{Haa78b} and \cite[Theorem 1.1 (c)]{Haa78b} the crossed product $\tilde{\cal{M}}$ is semifinite and there is a unique normal semifinite faithful trace $\tau=\tau_{\varphi}$ on $\tilde{\cal{M}}$ satisfying 
\begin{equation}
\label{trace-carac}
(\mathrm{D}\widehat{\varphi}:\mathrm{D} \tau)_t
=\lambda_t \ot \Id_H, \quad t \in \R
\end{equation}
where $(\mathrm{D}\widehat{\varphi}:\mathrm{D} \tau)_t$ denotes the Connes cocycle \cite[p.~48]{Str81} \cite[p.~111]{Tak03} of the dual weight $\widehat{\varphi}$ with respect to $\tau$. Moreover, $\tau$ satisfies the relative invariance $\tau \circ \widehat{\sigma}_s = \e^{-s}\tau$ for any $s \in \R$ by \cite[Lemma 5.2]{Haa78b}.

If $\psi$ is a normal semifinite weight on $\cal{M}$, we denote by $h_\psi$ the Pedersen-Takesaki derivative of the dual weight $\widehat{\psi}$ with respect to $\tau$ given by \cite[Theorem 4.10]{Str81}. By \cite[Corollary 4.8]{Str81}, note that the relation of $h_\psi$ with the Radon-Nikodym cocycle of $\widehat\psi$ is
\begin{equation}
\label{Radon-Nikodym-1}
(\D\widehat\psi:\mathrm{D}\tau)_t
= h_\psi^{\i t}, \quad t \in \R.
\end{equation}
If $\psi=\varphi$, we let $D_\varphi \ov{\mathrm{def}}{=} h_\varphi$ and we call it the density operator of $\varphi$.

By \cite[Chap. II, Prop. 4]{Ter81}, the mapping $\psi \to h_\psi$ gives a bijective correspondence between the set of all normal semifinite weights on $\cal{M}$ and the set of positive selfadjoint operators $h$ affiliated with $\tilde{\cal{M}}$ satisfying
\begin{equation}\label{eq:def-L1}
\widehat{\sigma}_s(h)
=\e^{-s}h, \quad s \in \R.
\end{equation}
Moreover, by \cite[Chap. II, Cor. 6]{Ter81}, $\omega$ belongs to $\cal{M}_*^+$ if and only if $h_\omega$ belongs to $\L^0(\tilde{\cal{M}},\tau)_+$. One may extend by linearity the map $\omega \mapsto h_\omega$ on $\cal{M}_*$. The Haagerup space $\L^1(\cal{M},\varphi)$ is defined as the set $\{h_\omega : \omega \in \cal{M}_*\}$, i.e. the range of the previous map. 

By \cite[Chap. II, Th.~7]{Ter81}, the mapping $\omega \mapsto h_\omega$, $\cal{M}_* \to \L^1(\cal{M},\varphi)$ is a linear order isomorphism which preserves the conjugation, the module, and the left and right actions of $\cal{M}$. Then $\L^1(\cal{M},\varphi)$ may be equipped with a continuous linear functional $\tr_\varphi \co \L^1(\cal{M}) \to \C$ defined by
\begin{equation}
\label{Def-tr}
\tr_\varphi(h_\omega)
\ov{\mathrm{def}}{=} \omega(1),\quad \omega \in \cal{M}_*
\end{equation}
\cite[Chap. II, Def. 13]{Ter81}. We also use the notation $\tr$ instead of $\tr_\varphi$.  A norm on $\L^1(\cal{M},\varphi)$ may be defined by $\norm{h}_1 \ov{\mathrm{def}}{=} \tr(|h|)$ for every $h \in \L^1(\cal{M},\varphi)$. By \cite[Chap. II, Prop. 15]{Ter81}, the map $\cal{M}_* \to \L^1(\cal{M},\varphi)$, $\omega \mapsto h_\omega$ is a surjective isometry.

More generally for $1 \leq p \leq \infty$, the Haagerup $\L^p$-space $\L^p(\cal{M},\varphi)$ associated with the normal faithful semifinite weight $\varphi$ is defined \cite[Chap. II, Def. 9]{Ter81} as the subset of the topological $*$-algebra $\L^0(\tilde{\cal{M}},\tau)$ of all (unbounded) $\tau$-measurable operators $x$ affiliated with $\cal{M}$ satisfying for any $s \in \R$ the condition
\begin{equation}
\label{Def-Haagerup}
\widehat{\sigma}_s(x)=\e^{-\frac{s}{p}}x \quad \text{if } p<\infty 
\quad \text{and} \quad \widehat{\sigma}_s(x)=x \quad \text{ if } p=\infty	
\end{equation}
where $\widehat{\sigma}_s \co \L^0(\cal{M},\tau) \to \L^0(\cal{M},\tau)$ is here the continuous $*$-automorphism obtained by a natural extension of the dual action \eqref{Dual-action} on $\cal{M}$. By \eqref{carac-Pi-de-M}, the space $\L^\infty(\cal{M},\varphi)$ coincides with $\pi(\cal{M})$ that we identify with $\cal{M}$. The spaces $\L^p(\cal{M},\varphi)$ are closed selfadjoint linear subspaces of $\L^0(\cal{M},\tau)$. They are closed under left and right multiplications by elements of $\cal{M}$. If $h=u|h|$ is the polar decomposition of $h \in \L^0(\tilde{\cal{M}},\tau)$ then by \cite[Chap. II, Prop. 12]{Ter81} we have
$$
h \in \L^p(\cal{M},\varphi)
\iff 
u \in \cal{M} \text{ and } |h| \in \L^p(\cal{M},\varphi).
$$

Suppose $1 \leq p<\infty$. By \cite[Chap.~II, Prop.~12]{Ter81} and its proof, for any $h \in \L^0(\tilde{\cal{M}},\tau)_+$, we have $h^p \in \L^0(\tilde{\cal{M}},\tau)_+$. Moreover, an element $h \in \L^0(\tilde{\cal{M}},\tau)$ belongs to $\L^p(\cal{M},\varphi)$ if and only if $|h|^p$ belongs to $\L^1(\cal{M},\varphi)$. A norm on $\L^p(\cal{M},\varphi)$ is then defined by the formula 
\begin{equation}
\label{Def-norm-Lp}
\norm{h}_p
\ov{\mathrm{def}}{=} (\tr |h|^p)^{\frac{1}{p}}
\end{equation} 
if $1 \leq p < \infty$ and by $\norm{h}_\infty \ov{\mathrm{def}}{=}\norm{h}_\cal{M}$, see \cite[Chap.~II, Def.~14]{Ter81}.

\paragraph{Case of a normal faithful positive linear form} If $\varphi$ is a normal faithful positive linear form on $\cal{M}$ then by \cite[(1.13)]{HJX10} the density operator $D_\varphi$ belongs to $\L^1(\cal{M},\varphi)$ and 
\begin{equation}
\label{varphi-trace}
\varphi(x)
=\tr_\varphi(D_\varphi x)
=\tr_\varphi(xD_\varphi), \quad x \in \cal{M}.
\end{equation}

\paragraph{Duality} Let $p, p^* \in [1,\infty]$ with $\frac{1}{p}+\frac{1}{p^*}=1$. By \cite[Chap.~II, Prop.~21]{Ter81}, for any $h \in \L^p(\cal{M},\varphi)$ and any $k \in \L^{p^*}(\cal{M},\varphi)$ the elements $hk$ and $kh$ belong to the space $\L^1(\cal{M},\varphi)$ and we have the tracial property $\tr(hk)=\tr(kh)$.

If $1 \leq p<\infty$, by \cite[Ch.~II, Th.~32]{Ter81} the bilinear form $\L^p(\cal{M},\varphi) \times \L^{p^*}(\cal{M},\varphi)\to \C$, $(h,k) \mapsto \tr(h k)$ defines a duality bracket between $\L^p(\cal{M},\varphi)$ and $\L^{p^*}(\cal{M},\varphi)$, for which $\L^{p^*}(\cal{M},\varphi)$ is isometrically the dual of $\L^p(\cal{M},\varphi)$. 


\paragraph{Change of weight} It is essentially proved in \cite[p.~59]{Ter81} and also \cite[Theorem 5.1]{Ray03}  
that $\L^p(\cal{M},\varphi)$ is independent of $\varphi$ up to an isometric isomorphism preserving the order and bimodular structure of $\L^p(\cal{M},\varphi)$, as well as the external products and Mazur maps. In fact given two normal semifinite faithful weights $\varphi_1, \varphi_2$ on $\cal{M}$ there is a $*$-isomorphism $\kappa \co \cal{M}_1 \to \cal{M}_2$ between the crossed products $\tilde{\cal{M}}_i\ov{\mathrm{def}}{=} \cal{M} \rtimes_{\sigma^\varphi_i} \R$ preserving $\cal{M}$, as well as the dual actions and the traces of $\tilde{\cal{M}}_1$ and $\tilde{\cal{M}}_2$, that is 
\begin{align}
\label{kappa}
\pi_2 = \kappa\circ \pi_1,\quad \widehat\sigma_2\circ\kappa = \kappa \circ \widehat\sigma_1 
\quad \text{and} \quad
\tau_2=\tau_1 \circ \kappa^{-1}.
\end{align}
Furthermore, $\kappa$ extends naturally to a topological $*$-isomorphism $\kappa \co \L^0(\tilde{\cal{M}}_1,\tau_1) \to \L^0(\tilde{\cal{M}}_2,\tau_2)$ between the algebras of measurable operators, which restricts to isometric $*$-isomorphisms between the noncommutative $\L^p$-spaces $\L^p(\cal{M},\varphi_1)$ and $\L^p(\cal{M},\varphi_2)$, preserving the $\cal{M}$-bimodule structures. Finally $\kappa \co \L^1(\cal{M},\varphi_1) \to \L^1(\cal{M},\varphi_2)$ preserves the traces:
\begin{equation}
\label{Trace-preserving}
\tr_{\varphi_1}
=\tr_{\varphi_2} \circ \kappa.
\end{equation}
Since $\kappa$ preserves the $p$-powers operations, i.e. $\kappa(h^p)=(\kappa(h))^p$ for any $h \in \L^0(\tilde{\cal{M}})$, it induces an isometry from $\L^p(\cal{M},\varphi_1)$ onto $\L^p(\cal{M},\varphi_2)$. It is not hard to see that this isometry is a completely order isomorphism.
 
This independence allows us to consider $\L^p(\cal{M},\varphi)$ as a particular realization of an abstract space $\L^p(\cal{M})$.


\paragraph{Centralizer of a normal faithful positive linear form} Recall that the centralizer \cite[p.~38]{Str81} of a  normal faithful positive linear form is the von Neumann subalgebra $\cal{M}^\varphi \ov{\mathrm{def}}{=} \{x \in \cal{M}: \sigma_t^\varphi(x)=x\text{ for all } t \in \R\}$. If $x \in \cal{M}$, we have by \cite[(2) p.~39]{Str81}
\begin{equation}
\label{Charac-centralizer}
x \in \cal{M}^\varphi
\iff \varphi(xy)=\varphi(yx) \text{ for any } y \in \cal{M}. 
\end{equation}

\paragraph{Reduced noncommutative $\L^p$-spaces} 


Let $\varphi$ be a normal faithful positive linear form. If the orthogonal projection $e$ \textit{belongs} to the centralizer of $\varphi$, we can consider the restriction $\varphi_e$ of $\varphi$ on $e\cal{M}e$. It is well-known that we can identify the Banach space $\L^p(e\cal{M}e,\varphi_e)$ with the subspace $e\L^p(\cal{M},\varphi)e$ of $\L^p(\cal{M},\varphi)$, see \cite[p.~508]{Wat88}. 
Moreover, we have the following result. 

\begin{lemma}
\label{Lemme-trace coincide}
The Haagerup trace $\tr_\varphi$ restricts to $\tr_{\varphi_e}$  on $\L^1(e\cal{M}e)$. 
\end{lemma}



Let $e$ be an orthogonal projection of $\cal{M}$. Let us construct a normal faithful positive linear form with centralizer containing $e$. 
Consider two normal faithful positive linear forms $\varphi_1$ and $\varphi_2$ on $e\cal{M}e$ and $e^\perp \cal{M} e^\perp$. By \cite[p.~155]{RaX03}, we can define a normal faithful positive linear form $\phi$ on $\cal{M}$ by 
\begin{equation}
\label{Extension-weight1}
\phi(x)
\ov{\mathrm{def}}{=} \varphi_1(exe)+\varphi_2(e^\perp x e^\perp), \quad x \in \cal{M}_+.
\end{equation}
Moreover, $e$ belongs to the centralizer of $\phi$ by \eqref{Charac-centralizer} and we have $\phi_e=\varphi_1$.


%
%

\paragraph{Matrix order and norms} We refer to \cite[Lemma A.2]{SkV19} and \cite[p.~155]{Hia21} for more information on the identification between $S^p_n(\L^p(\cal{M}))$ and $\L^p(\M_n(\cal{M}))$. The next lemma is a easy generalization of \cite[Lemma 2.13]{ArK23} (for semifinite von Neumann algebras) left to the reader using reduction theory \cite{HJX10}.

\begin{lemma}
\label{Lemma-Matricial-inequality}
Let $\cal{M}$ be a von Neumann algebra. Suppose $1 \leq p \leq \infty$. Let $a,b$ and $c$ be elements of $\L^p(\cal{M})$ such that the element
$
\begin{bmatrix}
a & b\\
b^* & c
\end{bmatrix}
$
of $S^p_2(\L^p(\cal{M}))$ is positive. Then we have $\norm{b}_{\L^p(\cal{M})} \leq \sqrt{\norm{a}_{\L^p(\cal{M})} \norm{c}_{\L^p(\cal{M})}}$.
\end{lemma}

The following is \cite[Lemma 2.5]{Han09}. 

\begin{lemma}
\label{Lemma-Han}
Suppose $\cal{M}$ is a von Neumann algebra equipped with a faithful semifinite normal weight $\varphi$. If $x$ is an element of the Haagerup noncommutative $\L^p$-space $\L^p(\cal{M})$, then we have 
$$
\begin{bmatrix} 
|x^*|& x \\ 
x^* & |x|
\end{bmatrix} 
\in \L^p(\M_2(\cal{M}))_+. 
$$
\end{lemma}

\paragraph{Extension of maps on noncommutative $\L^p$-spaces}
Let $\cal{N}$ be a von Neumann algebra equipped with a normal faithful linear form $\psi$. Consider a unital positive map $T \co \cal{N} \to \cal{N}$ such that $\psi \circ T =\psi$. Given $1 \leq p < \infty$, the map
\begin{equation}
\label{Map-extension-Lp}
\begin{array}{cccc}
T_p  \co &  D_\psi^{\frac{1}{2p}} \cal{N} h_\psi^{\frac{1}{2p}}  &  \longrightarrow   &  \L^p(\cal{N}) \\
           &   D_\psi^{\frac{1}{2p}}x D_\psi^{\frac{1}{2p}}  & \longmapsto &  D_\psi^{\frac{1}{2p}}T(x) D_\psi^{\frac{1}{2p}}  \\
\end{array}
\end{equation}
extends to a contractive map $T_p$ from $\L^p(\cal{N})$ into $\L^p(\cal{N})$. See \cite[Remark 5.6]{HJX10}.

The following is a folklore observation, see \cite[Lemma 2.2]{ArR19} for a proof.

\begin{lemma}
\label{lemma2-GL}
Let $\cal{M}$ be a von Neumann algebra and  $1 \leq p <\infty$. Let $h$ be a positive element of $\L^p(\cal{M})$.
\begin{enumerate}
	\item The map $s(h)\cal{M}s(h) \to \L^p(\cal{M})$, $x \mapsto h^{\frac{1}{2}} x h^{\frac{1}{2}}$ is injective.
	\item Suppose $1 \leq p <\infty$. The subspace $h^{\frac{1}{2}} \cal{M} h^{\frac{1}{2}}$ is dense in $s(h)\L^p(\cal{M})s(h)$ for the topology of $\L^p(\cal{M})$.
\end{enumerate}
\end{lemma}

\paragraph{Normalized duality mappings} Recall that a normed linear space $X$ is said to be strictly convex (or rotund) if for any $x,y \in X$ the equalities $\frac{\norm{x+y}_X}{2}=\norm{x}_X=\norm{y}_X$ imply $x=y$.

Let $X$ be a Banach space. For each $x \in X$, we can associate \cite[Definition 2.12]{Pat18} the subset
\begin{equation}
\label{Def-JX}
J_X(x) 
\ov{\mathrm{def}}{=} \big\{x^* \in X^*\ : \ \langle x, x^* \rangle_{X,X^*}= \norm{x}_X^2 = \norm{x^*}_{X^*}^2\big\}
\end{equation}
of the dual $X^*$. 

The multivalued operator $J_X \co X \to X^*$ is called the normalized duality mapping of $X$. From the Hahn-Banach theorem, for every $x \in X$, there exists $y^* \in X^*$ with $\norm{y^*}_{X^*} = 1$ such that $\langle x, y^*\rangle_{X,X^*} = \norm{x}_X$. Using $x^*=\norm{x}_Xy^*$, we conclude that $J_X(x)\not= \emptyset$ for each $x \in X$. If the dual space $X^*$ is strictly convex, $J_X$ is single-valued. 

When $X$ is a reflexive strictly convex Banach space with a strictly convex dual space $X^*$, $J_X$ is a single-valued bijective map and its inverse $J_X^{-1} \co X^* \to X^{**}=X$ is equal to $J_{X^*} \co X^* \to X$. 


If the Banach space $X$ is a noncommutative $\L^p$-space, we have the following explicit description of the normalized duality mapping, see \cite{ArR19} for a proof.

\begin{lemma}
\label{Lemma-J-lp}
Suppose $1<p<\infty$. If $h$ belongs to $\L^p(\cal{M})$ with polar decomposition $h=u|h|$ then we have
\begin{align}
\label{comput-J}
J_{\L^p(\cal{M})}(h)
=\norm{h}^{2-p}_{p}|h|^{p-1}u^*.
\end{align}
\end{lemma}

The next crucial result is proved in \cite{Arh20}.

\begin{lemma}
\label{Lemma-smooth}
Let $X$ be a smooth strictly convex reflexive Banach space. Let $P \co X \to X$ be a contractive projection and $x$ be an element of $X$. Then $x$ belongs to $\Ran P$ if and only if $J_X(x)$ belongs to $\Ran P^*$.
\end{lemma}

\paragraph{Projections and interpolation} 

We start by recalling some background on complex interpolation theory. We refer to the books \cite{BeL76}, \cite{KPS82} and \cite{Lun18} for more information. Let $Y_0,Y_1$ be two Banach spaces which embed into a topological vector space space $\tilde{Y}$. We say that $(Y_0,Y_1)$ is an interpolation couple. Then the sum $Y_0 + Y_1 \ov{\mathrm{def}}{=} \{y \in \tilde{Y} : y=y_1+y_2 \text{ for some } y_1 \in Y_1,y_2 \in Y_2\}$ is well-defined and equipped with the norm
$$
\norm{y}_{Y_0+Y_1}
\ov{\mathrm{def}}{=} \inf_{y=y_0+y_1} \big(\norm{y_0}_{Y_0}+\norm{y_1}_{Y_1}\big).
$$
The intersection $Y_0 \cap Y_1$ is equipped with the norm 
\begin{equation}
\label{norm-intersection}
\norm{y}_{Y_0 \cap Y_1}
\ov{\mathrm{def}}{=} \max\{\norm{y}_{Y_0},\norm{y}_{Y_1}\}, \quad y \in Y_0 \cap Y_1.
\end{equation}
Consider the closed strip $\ovl{S} \ov{\mathrm{def}}{=} \{z \in \mathbb{C} : 0 \leq \Re z \leq 1\}$. Let us denote by $\mathscr{F}(Y_0,Y_1)$ the family of bounded continuous functions $f \co \ovl{S} \to Y_0+Y_1$, holomorphic on the open strip $S=\{z \in \mathbb{C} : 0< \Re z <1\}$ inducing continuous functions $ \R\to Y_0$, $t \mapsto f(\i t)$ and $\R \to Y_1$, $t \mapsto f(1+\i t)$ which tend to 0 when $|t|$ goes to $\infty$. For any $f \in \mathscr{F}(Y_0,Y_1)$, we set
$$
\norm{f}_{\mathscr{F}(Y_0,Y_1)}
\ov{\mathrm{def}}{=} \max \left\{\sup_{t \in \R} \norm{f(\i t)}_{Y_0}, \sup_{t \in \R} \norm{f(1+\i t)}_{Y_1}\right\}.
$$
If $0 \leq \theta \leq 1$, we define the subspace $
(Y_0,Y_1)_\theta
\ov{\mathrm{def}}{=} \big\{f(\theta) : f \in \mathscr{F}(Y_0,Y_1)\big\}$ of the Banach space $Y_0+Y_1$. For any $y \in (Y_0,Y_1)_\theta$, we let
$$
\norm{y}_{(Y_0,Y_1)_\theta}
\ov{\mathrm{def}}{=} \inf \big\{\norm{f}_{\mathscr{F}(Y_0,Y_1)} : f \in \mathscr{F}(Y_0,Y_1), f(\theta)=y \big\}.
$$
Then by \cite[Theorem 4.1.2]{BeL76} $(Y_0,Y_1)_\theta$ equipped with this norm is a Banach space.

We will use the following well-known result. See e.g. \cite[Lemma 4.8]{ArK23} for a slightly more general statement.

\begin{lemma}
\label{Lemma-interpolation}
Let $(Y_0,Y_1)$ be an interpolation couple and let $C$ be a contractively complemented subspace of $Y_0+Y_1$. We assume that the corresponding contractive projection $P \co Y_0+Y_1 \to Y_0+Y_1$ satisfies $P(Y_i) \subset Y_i$ and that the restriction $P \co Y_i \to Y_i$ is contractive for $i=0,1$. Then $(Y_0 \cap C,Y_1 \cap C)$ is an interpolation couple and the canonical inclusion $J \co C \to Y_0+Y_1$ induces an isometric isomorphism $\tilde{J}$ from $(Y_0 \cap C,Y_1 \cap C)_\theta$ onto the subspace $P((Y_0,Y_1)_\theta)=(Y_0,Y_1)_\theta \cap C$ of $(Y_0,Y_1)_\theta$. 
\end{lemma}

\paragraph{Strictly monotone ordered Banach spaces} An ordered Banach space $X$ (or its norm) is said to be monotone if $0 \leq x \leq y$ implies $\norm{x} \leq \norm{y}$ for any $x,y \in X$. A monotone ordered Banach space $X$ is said to be strictly monotone if $0 \leq x \leq y$ and $x \not=y$ implies $\norm{x}_X < \norm{y}_X$. We will use the following crucial lemma of \cite{Arh20}. For the sake of completeness, we give the following short argument. 

\begin{lemma}
\label{Lemma-monotone}
Let $X$ be a strictly monotone ordered Banach space. Let $T \co X \to X$ be a contraction. Consider $x,y \in X$ such that $0 \leq x \leq y$ with $T(y)=y$ and $T(x)=0$. Then $x=0$.
\end{lemma}

\begin{proof}
We have
$$
y
=T(y)
=T(y)-T(x)
=T(y-x).
$$ 
Since $T$ is contractive, we deduce that $\norm{y}_X =\norm{T(y)-T(x)}_X= \norm{T(y-x)}_X \leq \norm{y-x}_X$. Since $0 \leq y-x \leq y$ we infer that $\norm{y}_X=\norm{y-x}_X$ and finally $x=0$ by strict monotonicity of the norm.
\end{proof}

\section{Preliminaries on TRO's, Jordan algebras and $\JBW^*$-triples}
\label{Sec-prelim-TRO}


We start by providing a little overview of the theory of ternary rings of operators.

\paragraph{Ternary rings of operators} A ternary ring of operators (or simply $\TRO$) is a norm closed subspace $V$ of the space $\B(H,K)$ for some Hilbert spaces $H$ and $K$ which is closed under the triple product $(x,y,z) \mapsto xy^*z$, see e.g. \cite[4.4.1 p.~161]{BLM04}. A TRO $V$ is called a $\W^*$-$\TRO$ if it is weak* closed in the dual Banach space $\B(H,K)$. A sub-TRO of a TRO $V$ is a closed subspace $W$ of $V$ satisfying $W W^*W \subset W$. We refer to the references given in Section \ref{Introduction} for more information on TROs. 

\begin{example} \normalfont
\label{typical-TRO}
A basic example of $\W^*$-$\TRO$ is given by $e\cal{M}f$ where $e,f$ are orthogonal projections of a von Neumann algebra $\cal{M}$.
\end{example}

\paragraph{Operator space structures of TROs}
We also note that every TRO admits an operator space structure. Indeed, let us assume that $V$ is a TRO contained in $\B(H,K)$. Then for each $n \in \mathbb{N}$, the matrix space $\M_n(V)$ can be identified with a TRO contained in $\M_n(\B(H,K)) = \B(H^n,K^n)$. This provides a canonical operator space matrix norm on $V$ such that each $\M_n(V)$ is again a TRO. By \cite[Proposition 2.1]{KaR02}, the TRO-matrix norms are uniquely determined on each TRO and does not depend on the choice of the representing Hilbert spaces.

\paragraph{$\TRO$-homomorphisms} A $\TRO$-homomorphism (or triple morphism) is a linear map $T \co V \to W$ between two $\TRO$s respecting the ternary product: 
$$
T(xy^*z) 
= T(x)T(y)^*T(z), \quad x,y,z \in V.
$$
If in addition, $T$ is an injection from $V$ onto $W$, we call $T$ a $\TRO$-isomorphism from $V$ onto $W$. By \cite[Proposition 2.4]{EOR01}, a TRO-isomorphism between $\W^*$-TROs is necessarily weak* continuous. By  \cite[Proposition 2.1]{EOR01}, every TRO-homomorphism is completely contractive, and every injective $\TRO$-homomorphism is completely isometric. Finally, if $T \co V \to W$ is a linear map between TRO's then by \cite[Proposition 2.1]{Ham99} (see also \cite[Corollary 4.4.6 p.~163]{BLM04}) $T$ is a surjective complete isometry if and only if $T$ is a surjective 2-isometry if and only if $T$ is a TRO-isomorphism.

\begin{example} \normalfont
\label{E-finite-dim-TRO}
Every finite-dimensional $\TRO$ $V$ is completely isometric (i.e. triple isomorphic) to a finite direct sum of rectangular matrix algebras, i.e. it has the form
$$
V 
= \M_{m_1,n_1} \oplus_\infty \cdots \oplus_\infty \M_{m_k,n_k}, 
$$
see \cite{EOR01}, \cite[Corollary A.2]{Kan13} and \cite{Smi00}.
\end{example}


\paragraph{Linking algebra} If $\cal{M}$ is a von Neumann algebra and $e$ is an orthogonal projection in $\cal{M}$, then $e\cal{M}(1- e)$ is a $\W^*$-TRO by Example \ref{typical-TRO}. Conversely, if the subspace $V$ of $\B(H,K)$ is a $\W^*$-TRO, then we can consider the subspace $V^* \ov{\mathrm{def}}{=} \{x^* : x \in V \}$ of $\B(K,H)$, the von Neumann subalgebras $M(V) \ov{\mathrm{def}}{=} \ovl{VV^*}^{\w^*}$, $N(V)\ov{\mathrm{def}}{=} \ovl{V^*V}^{\w^*}$ of $\B(K)$ and $\B(H)$ and the von Neumann algebra
\begin{equation}
\label{Linking-algebra}
R(V)
\ov{\mathrm{def}}{=} \begin{bmatrix}
  M(V)   &  V \\
  V^*   &  N(V) \\
\end{bmatrix}
\subset \B(K \oplus H)
\end{equation}
which is called the linking von Neumann algebra of $V$, see \cite[p.~846]{Rua04}. Then there exists a TRO-isomorphism $V = eR(V)e^\perp$, where $e 
\ov{\mathrm{def}}{=} \begin{bmatrix}
    \Id_K & 0  \\
    0 & 0  \\
\end{bmatrix}$ 
and 
$e^\perp
\ov{\mathrm{def}}{=} \begin{bmatrix}
   0  &  0 \\
   0  &  \Id_H \\
\end{bmatrix}$.

\paragraph{Completely contractive projections on TROs} 
We will use the following result \cite[Theorem 4.4.9 p.~165]{BLM04} (it is essentialy a slight generalization of a result of Youngson \cite{You83}).

\begin{thm}
\label{Th-youngson}
Let $V$ be a $\TRO$ and $P \co V \to V$ be a completely contractive projection. For any $x, y, z \in V$, we have  
$$
P\big(P(x)P(y)^*P(z)\big) 
= P(xP(y)^*P(z)) 
= P(P(x)y^*P(z)) 
= P(P(x)P(y)^*z).
$$
Furthermore, the range $\Ran(P)$ of $P$ is a $\TRO$ with triple product $(x,y,z) \mapsto P(xy^*z)$.
\end{thm}


Let $W$ be a sub-TRO of a TRO $V$. A projection $P \co V \to V$ on a sub-TRO $W$ is a TRO-conditional expectation on $W$ \cite[p.~499]{EOR01} if we have
\begin{equation}
\label{TRO-cond-exp}
P(zx^*y) 
= P(z)x^*y,\quad 
P(xz^*y) 
= xP(z)^*y, \quad 
P(xy^*z) 
= xy^*P(z), \quad z \in V, x, y \in W.
\end{equation}

\begin{remark} \normalfont
\label{Rem-TRO-cond}
Let $P \co V \to V$ be a projection from $V$ onto a sub-TRO $W$. Then by \cite[Theorem 2.5]{EOR01} (see also \cite[Corollary 4.4.10 p.~166]{BLM04}) the following properties are equivalent:
\begin{enumerate}
\item $P$ is completely contractive,
\item $P$ is contractive,
\item $P$ is a contractive conditional expectation.
\end{enumerate}
\end{remark}


Now, we give a brief presentation of $\JBW^*$-algebras, since these algebras provide examples of $\JBW^*$-triples (see Example \ref{Ex-JBW-algebra-triple}) and since the $\L^p$-spaces introduced in Section \ref{Sec-Lp-JBW} generalize the nonassociative $\L^p$-spaces of $\JBW^*$-algebras defined in the paper \cite{Arh23}.

\paragraph{Various Jordan algebras} A Jordan algebra $A$ over a field $\K$ is a vector space $A$ over $\K$ equipped with a commutative bilinear product that satisfies $(x^2 \circ y)\circ x=x^2\circ (y \circ x)$ for any $x,y \in A$. 
A Jordan algebra $A$ over $\R$ is called formally real \cite[p.~69]{HOS84} if for any $x_1,\ldots,x_n\in A$ the relation $x_1^2+\cdots +x_n^2=0$ implies $x_1=\cdots=x_n=0$. Following \cite[Definition 1.5 p.~5]{AlS03}, a $\JB$-algebra is a Jordan algebra over $\R$ with identity element 1 equipped with a complete norm satisfying the properties $\norm{x \circ y} \leq \norm{x} \norm{y}$, $\norm{x^2}=\norm{x}^2$, $\norm{x^2} \leq \norm{x^2+y^2}$ for any $x,y \in A$. A $\JBW$-algebra is a $\JB$-algebra which is a dual Banach space \cite[p.~111]{HOS84}. In this case, the predual is unique. Recall that a $\JBW$-algebra is always unital by \cite[Lemma 4.1.7]{HOS84}.

\begin{example} \normalfont
If $\O$ is the algebra of octonions, then the space 
$$
\mathrm{H}_3(\O)
=\left\{\begin{bmatrix}
   a  & \alpha & \beta  \\
   \ovl{\alpha}  & b & \gamma	\\
    \ovl{\beta} & \ovl{\gamma} &  c  \\
\end{bmatrix}: \alpha,\beta,\gamma \in \O, a,b,c \in \R \right\}
$$ 
of hermitian 3x3 matrices with entries in $\O$ equipped with the product $(x,y) \mapsto x \circ y=\frac{1}{2}(xy+yx)$ 
is a unital formally real Jordan algebra by \cite[Proposition 2.9.2 p.~69]{HOS84} of dimension 27. By \cite[Corollary 3.1.7 p.~77]{HOS84} and its proof, we can equip $\mathrm{H}_3(\O)$ with a norm that makes it a $\JB$-algebra. With this structure, $\mathrm{H}_3(\O)$ is a $\JBW$-algebra.
\end{example}

A $\JB^*$-algebra \cite[p.~91]{HOS84} \cite[Definition 3.3.1]{CGRP14} is a complex Banach space $A$ which is a complex Jordan algebra equipped with an involution satisfying 
\begin{equation}
\label{def-JBstar}
\norm{x \circ y} \leq \norm{x} \norm{y},\quad  \norm{x^*}=\norm{x} 
\quad \text{and} \quad \norm{\{x,x^*,x\}}=\norm{x}^3
\end{equation}
for any $x,y \in A$, where we use the Jordan triple product 
$$
\{x,y,z\}
\ov{\mathrm{def}}{=} (x \circ y) \circ z+(y \circ z)\circ x-(x \circ z) \circ y.
$$
A $\JBW^*$-algebra \cite[p.~4]{CGRP18} is a $\JB^*$-algebra which is a dual Banach space. For the links between $\JBW^*$-algebras and $\JBW$-algebras, we refer to \cite[pp.~4-5]{BHK17} and \cite[Corollary 5.1.29 p.~9 and Corollary 5.1.41 p.~15]{CGRP18}. In short, the $\JBW$-algebras are exactly the selfadjoint parts of $\JBW^*$-algebras. 

\begin{example} \normalfont
\label{Ex-von-Neumann-algebra}
A von Neumann algebra $\cal{M}$ equipped with the Jordan product
\begin{equation}
\label{Jordan-product}
x \circ y 
\ov{\mathrm{def}}{=} \frac{1}{2}(xy+yx), \quad x,y \in \cal{M} 
\end{equation}
is a $\JBW^*$-algebra.
\end{example}

\paragraph{$\JB^*$-triples and $\JBW^*$-triples} 
We recall that an element $a$ in a Banach algebra $A$ is called hermitian \cite[p.~156]{Chu12} if $\norm{\exp \i ta} = 1$ for all $t \in \R$.

A $\JB^*$-triple, e.g. \cite[p.~130]{CGRP14} \cite[Definition 2.5.25 p.~164]{Chu12}, is a complex Banach space $X$ equipped with a continuous triple product $X \times X \times X \to X$, $(x,y,z) \mapsto \{x,y,z\}$ such that
\begin{enumerate}
\item $\{x,y,z\}$ is bilinear, symmetric in $x$ and $z$ and conjugate linear in $y$,

\item $\{a,b,\{x,y,z\}\} = \{\{a,b,x\},y,z\} - \{x,\{b,a,y\},z\} + \{x,y,\{a,b,z\}\}$ for any $x,y,z,a,b \in X$,

\item the left multiplication $D(x,x) \co X \to X$, $y \mapsto \{x,x,y\}$ is Hermitian in the algebra $\B(X)$ of bounded operators on $X$ and has positive spectrum,

\item for any $x \in X$, we have $\norm{\{x,x,x\}}_X = \norm{x}_X^3$.
\end{enumerate}
The last axiom is a Jordan analogue of the Gelfand-Naimark axiom of $\mathrm{C}^*$-algebras. Moreover, if $X$ is a dual Banach space, then it is called a $\JBW^*$-triple \cite[Definition 2.5.30]{Chu12} \cite[p.~528]{CGRP14}. In this case, the predual is unique by \cite[Theorem 5.7.38 p.~233]{CGRP18}. 
We refer to the books \cite{Chu12}, \cite{CGRP14}, \cite{CGRP18}, \cite{Isi19} and \cite{Upm85} and to the important papers \cite{Hor87a}, \cite{Hor87b} and \cite{HoN88} for more information on $\JB^*$-triples and $\JBW^*$-triples.

\begin{example} \normalfont
\label{Ex-JBW-algebra-triple}
A $\JBW^*$-algebra $\cal{M}$ admits a structure of $\JBW^*$-triple with triple product 
\begin{equation}
\label{triple-prod-JBWstar}
(x,y,z) \mapsto \{x,y,z\} 
\ov{\mathrm{def}}{=} (x \circ y^*) \circ z + (z \circ y^*) \circ x - (x \circ z) \circ y^*.
\end{equation}
See \cite[Lemma 3.1.6 p.~174]{Chu12} and \cite[p.~224]{CGRP18}. 
\end{example}

\paragraph{Triple homomorphisms} A linear map $T \co X \to Y$ between two $\JB^*$-triples $X$ and $Y$ is called a triple homomorphism if it preserves the triple product:
$$
T(\{x,y,z\}) 
= \{T(x), T (y), T(z)\}, \quad x,y, z \in X.
$$
A bijective triple homomorphism is called a triple isomorphism. In $\JB^*$-triples, the metric structure and the algebraic structure determine each other. Indeed, by \cite[Theorem 3.1.7 p.~175]{Chu12} and \cite[Theorem 3.1.20 p.~183]{Chu12} a linear bijection between two $\JB^*$-triples is an isometry if and only if it is a triple-isomorphism. In this case, this bijection is necessarily weak* continuous, see \cite[Corollary 5.7.39 p.~234]{CGRP18}.

\paragraph{$\JC^*$-triples and $\JW^*$-triples} A $\JC^*$-triple (or $\mathrm{J}^*$-algebra) \cite[Definition 2.5.34 p.~169]{Chu12} \cite[p.~440]{CGRP18} \cite[Definition 22.7.1 p.~460]{Isi19} is a norm closed subspace of $\B(H,K)$ which is also closed under the map $x \mapsto xx^*x$. By a polarization argument, any $\JC^*$-triple is closed under the triple product $(x,y,z) \mapsto \{x,y,z\}\ov{\mathrm{def}}{=}\frac{1}{2}(xy^*z+zy^*x)$. A $\JC^*$-triple is a $\JB^*$-triple. Finally, a $\JC^*$-triple is called a $\JW^*$-triple if it is a dual Banach space. Consequently, a $\JW^*$-triple is a $\JBW^*$-triple. We shall often say that a $\JBW^*$-triple is a $\JW^*$-triple if it is isometrically isomorphic to a $\JW^*$-triple.

\begin{example} \normalfont
A $\W^*$-$\TRO$ is a $\JW^*$-triple with triple product $(x,y,z) \mapsto \{x,y,z\} \ov{\mathrm{def}}{=} xy^*z$.
\end{example}

\begin{example} \normalfont
By \cite[Proposition 6.1.41 p.~362]{CGRP18}, the complexification $\mathrm{H}_3(\O)_\C=\mathrm{H}_3(\O_\C)$ of the $\JBW$-factor $\mathrm{H}_3(\O)$ is equipped with a structure of $\JBW^*$-factor. The underlying $\JBW^*$-triple is called a Cartan factor of type VI.
\end{example}

\begin{example} \normalfont
By \cite[p.~140]{Isi19}, the space $\Asym_n \ov{\mathrm{def}}{=} \{x \in \M_n : x^t=-x \}$ of skew-symmetric complex matrices of $\M_n$ is a $\JBW^*$-triple (called a Cartan factor of type $\III_n$). If $n$ is even, it is even equipped with a structure of reversible $\JW^*$-algebra by \cite[Proposition 25.2.2 p.~513]{Isi19}.
\end{example}

\paragraph{Projections on $\JBW^*$-triples} We will use the following, which is \cite[Theorem 5.6.59 p.~199]{CGRP18} (see also \cite[Theorem 14.4.1 p.~252]{Isi19} and \cite[Theorem 3.3.1 p.~202]{Chu12}). This is a fundamental result of Kaup \cite{Kau84} and Stach\'o \cite{Sta82}. The last sentence is from \cite{FrB84} and \cite{FrB85}. See also \cite{ChE77}, \cite{EfS79} and \cite{RoY82} for particular cases.

\begin{thm}
\label{Th-proj-JBW}
If $P \co X \to X$ is a weak* continuous contractive projection on a $\JBW^*$-triple $X$, then the range $P(X)$ is a $\JBW^*$-triple with the triple product given by $\{x,y,z\}_P \ov{\mathrm{def}}{=} P(\{x,y,z\})$ where $x,y,z \in P(X)$. Moreover, we have
$$
P\big\{P(x), y,P(z)\big\} 
= P\big\{P(x),P(y),P(z)\big\},\quad x,y,z \in X.
$$
Finally, if $X$ is a $\JW^*$-triple then $P(X)$ is a $\JW^*$-triple.
\end{thm}
In general, the range $P(X)$ is not a subtriple of $X$ (see \cite[Example 1 p.~66]{FrB82} or \cite[Example 3 p.~99]{Kau84}). But note that if $P(X)$ is known to be a subtriple then the triple product $\{\cdot,\cdot,\cdot\}_{P}$ coincides with the original triple product of $X$ because in $\JB^*$-triples norm and triple product determine each other (see e.g. \cite[Theorem 3.1.7 p.~175 and Theorem 3.1.20 p.~183]{Chu12}). 

\paragraph{Tripotents and Peirce projections} An element $u$ in a $\JBW^*$-triple $X$ satisfying $\{u,u,u\} = u$ is called a tripotent. When $X$ is a $\JW^*$-triple, these elements are precisely the partial isometries of $X$. With a tripotent $u$ and $0 \leq i \leq 2$, we can introduce the Peirce projections $P_i(u) \co X \to X$ with range $X_i(u)$. For any $x \in X$, we have
\begin{equation}
\label{Peirce-1}
P_2(u)(x) 
\ov{\mathrm{def}}{=} \{u, \{u, x, u\},u\}, \quad
P_1(u) \ov{\mathrm{def}}{=} 2\big(D(u, u) - P_2(u)\big) 
\end{equation}
and 
\begin{equation}
\label{Peirce-2}
P_0(u)+ P_1(u)+ P_2(u) 
= \Id_{X}.
\end{equation}
These maps are contractive linear projections by \cite[Corollary 1.2]{FrB85b}. A crucial property of $\JBW^*$-triples is that for a tripotent $u$ of $X$ the Peirce-2 subspace $X_2(u)$ is a $\JBW^*$-algebra with product $(a,b) \mapsto a \circ b \ov{\mathrm{def}}{=} \{a,u,b\}$, involution $a^* \ov{\mathrm{def}}{=} \{u, a, u\}$ where $a,b \in X_2(u)$, and unit $u$.

\begin{example} \normalfont
\label{Pierce-example}
In the case where $X$ is a $\JBW^*$-algebra and $e \in X$ is a projection, i.e. a selfadjoint idempotent, the Peirce projections are given by the following expressions (essentially \cite[p.~48]{HOS84})
\begin{equation}
\label{Pierce-example-fo}
P_2(e)x=\{e, x^*, e\}, \quad P_1(e)x= 2\{e,x^*,1-e\}, \quad P_0(e)x=\{1-e,x^*,1-e\},
\end{equation}
where $x \in X$.
\end{example}

A tripotent $u$ in a $\JBW^*$-triple $X$ is called complete if $X_0(u) = \{0\}$ \cite[p.~16]{HKPP20} \cite[p.~517]{CGRP14} and unitary if $X = X_2(u)$. 

\paragraph{Orthogonality and order} Two tripotents $u$ and $v$ are said orthogonal if $v$ belongs to $X_0(u)$. This relation is symmetric on the set of tripotents. For two tripotents $u$ and $v$, we write $u \leq v$ if $v-u$ is a tripotent orthogonal to $u$. The relation $\leq$ is a partial order on the set of tripotents by \cite[Theorem 2.1]{EdR98a}.

\begin{example} \normalfont
\label{order-algebras}
Let $\cal{M}$ be a $\JBW^*$-algebra. Each projection (i.e. selfadjoint idempotent) $p$ of $\cal{M}$ is a tripotent. Moreover, for any projections $e$, $f$ in $\cal{M}$, we have $e \leq f$ if and only if $e\circ f =e$ and $e \perp f$ if and only if $e \circ f = 0$.
\end{example}

We have the following characterization of triple order (\cite[Proposition 6.9]{BHKPP18}).

\begin{prop}
\label{Lemma-smaller}
Let $X$ be a $\JBW^*$-triple and let $u,v$ be two tripotents of $X$. Then the following assertions are equivalent:
\begin{enumerate}
\item $u \leq v$,
\item $u= \{ u,v,u\}$,
\item $u$ is a projection in the $\JBW^*$-algebra $X_2(v)$.
\end{enumerate}
\end{prop}

\section{$n$-pseudo-decomposable maps on noncommutative $\L^p$-spaces}
\label{Sec-pseudo}

In this section, we define and give several properties of $n$-pseudo-decomposable maps on Haagerup noncommutative $\L^p$-spaces. The proofs are (essentially) similar to the ones of \cite{ArK23} provided for results on decomposable maps on tracial noncommutative $\L^p$-spaces. The author hesitated to choose the terminology <<$n$-decomposable operator>>. Since there are already two different notions for the name <<decomposable operator>> (the one of  \cite{Haa85} \cite{JuR04} \cite{ArK23} and the one of \cite[Definition 1.2.8 p.~7]{Sto13}), we chose <<$n$-pseudo-decomposable>>.

\begin{defi}
\label{Def-pseudo-decomposable}
Suppose $1 \leq p \leq \infty$. Let $\cal{M}$ and $\cal{N}$ be von Neumann algebras. Consider some $n \in \{1,2,\ldots,\infty\}$. A linear map $T \co \L^p(\cal{M}) \to \L^p(\cal{N})$ is $n$-pseudo-decomposable if there exist some bounded linear maps $v_1,v_2 \co \L^p(\cal{M}) \to \L^p(\cal{N})$ such that the linear map
\begin{equation}
\label{Matrice-2-2-Phi}
\Phi
\ov{\mathrm{def}}{=}\begin{bmatrix}
   v_1  &  T \\
   T^\circ  &  v_2  \\
\end{bmatrix}
\co S^p_2(\L^p(\cal{M})) \to S^p_2(\L^p(\cal{N})), \quad \begin{bmatrix}
   a  &  b \\
   c &  d  \\
\end{bmatrix}\mapsto 
\begin{bmatrix}
   v_1(a)  &  T(b) \\
   T^\circ(c)  &  v_2(d)  \\
\end{bmatrix}
\end{equation}
is $n$-positive, that means that $\Id_{S^p_n} \ot \Phi \co S^p_n(S^p_2(\L^p(\cal{M}))) \to S^p_n(S^p_2(\L^p(\cal{N})))$ is a positive map, where $T^\circ(c) \ov{\mathrm{def}}{=} T(c^*)^*$. In this case, we let
\begin{equation}
\label{Norm-dec}
\norm{T}_{n-\pdec,\L^p(\cal{M}) \to \L^p(\cal{N})}
\ov{\mathrm{def}}{=}\inf \max\{\norm{v_1},\norm{v_2}\}
\end{equation}
where the infimum is taken over all maps $v_1$ and $v_2$. We say that $T$ is contractively $n$-pseudo-decomposable if $\norm{T}_{n-\pdec,\L^p(\cal{M}) \to \L^p(\cal{N})} \leq 1$. 
\end{defi}
If $n \in \{1,2,\ldots,\infty\}$, we denote by $\PDec_n(\L^p(\cal{M}),\L^p(\cal{N}))$ the set of $n$-pseudo-decomposable operators between two noncommutative $\L^p$-spaces. 

If $n=\infty$, we recover the decomposable maps of \cite{JuR04} and of the memoir \cite{ArK23} (and the decomposable maps of \cite{Haa78b} if in addition $p=\infty$) and we have $\norm{T}_{\infty-\pdec,\L^p(\cal{M}) \to \L^p(\cal{N})}=\norm{T}_{\dec,\L^p(\cal{M}) \to \L^p(\cal{N})}$. By \cite[Theorem 3.23]{ArK23}, if the von Neumann algebras $\cal{M}$ and $\cal{N}$ are approximately finite-dimensional and equipped with normal semifinite faithful traces, the decomposable norm $\norm{\cdot}_{\dec,\L^p(\cal{M}) \to \L^p(\cal{N})}$ and the regular norm $\norm{\cdot}_{\reg,\L^p(\cal{M}) \to \L^p(\cal{N})}$ of \cite{Pis95} are identical. We will show that the infimum \eqref{Norm-dec} is a minimum (see Proposition \ref{Prop-dec-inf-atteint}). Note that in the conditions of Definition \ref{Def-pseudo-decomposable} the maps $v_1$ and $v_2$ are $n$-positive.

Let $\cal{M}_1$, $\cal{M}_2$ and $\cal{M}_3$ be von Neumann algebras. Suppose $1 \leq p \leq \infty$. Let $T_1 \co \L^p(\cal{M}_1) \to \L^p(\cal{M}_2)$ and $T_2 \co \L^p(\cal{M}_2) \to \L^p(\cal{M}_3)$ be some $n$-pseudo-decomposable maps. It is easy to see
that the composition $T_2 \circ T_1$ is $n$-pseudo-decomposable and that 
\begin{equation}
\label{Composition-dec}
\norm{T_2 \circ T_1}_{n-\pdec} \leq \norm{T_2}_{n-\pdec} \norm{T_1}_{n-\pdec}.
\end{equation}

Recall that for any matrix $\alpha \in \M_{n,m}$, the map  
\begin{equation}
\label{conj-cp}
\L^p(\M_n(\cal{M})) \to \L^p(\M_m(\cal{M})), \quad x \mapsto \alpha^* x \alpha
\end{equation}
is completely positive. 
\begin{prop}
\label{prop-dec-homogeneous}
Let $\cal{M}$ and $\cal{N}$ be von Neumann algebras. Suppose $1 \leq p\leq \infty$. Consider some $n \in \{1,2,\ldots,\infty\}$. If $\lambda \in \C$ and if $T \co \L^p(\cal{M}) \to \L^p(\cal{N})$ is $n$-pseudodecomposable then the map $\lambda T$ is $n$-pseudodecomposable and $\norm{\lambda T}_{n-\pdec,\L^p(\cal{M}) \to \L^p(\cal{N})}=|\lambda|\, \norm{T}_{n-\pdec,\L^p(\cal{M}) \to \L^p(\cal{N})}$.
\end{prop}

\begin{proof}
By symmetry, it suffices to prove $\norm{\lambda T}_{n-\pdec} \leq |\lambda| \norm{T}_{n-\pdec}$, since then $\norm{T}_{n-\pdec} = \bnorm{\frac{1}{\lambda} \lambda T}_{n-\pdec} \leq \frac{1}{|\lambda|} \norm{\lambda T}_{n-\pdec}$. We can write $\lambda=|\lambda| \theta$ where $\theta$ is a complex number such that $|\theta|=1$. Assume that $v_1,v_2 \co \L^p(\cal{M}) \to \L^p(\cal{N})$ are linear maps such that the map 
$\begin{bmatrix} 
v_1 & T \\ 
T^\circ & v_2
\end{bmatrix} 
\co S^p_2(\L^p(\cal{M})) \to S^p_2(\L^p(\cal{N}))$ is $n$-positive. By \eqref{conj-cp}, the linear map 
$$
\begin{bmatrix} 
1 & 0 \\ 
0 & \theta 
\end{bmatrix} ^* 
\begin{bmatrix} 
v_1(\cdot) & T(\cdot) \\ 
T^\circ(\cdot) & v_2(\cdot) 
\end{bmatrix} 
\begin{bmatrix} 
1 & 0 \\ 
0 & \theta 
\end{bmatrix} 
$$ 
is also $n$-positive on $S^p_2(\L^p(\cal{M}))$ by composition of $n$-positive maps. But it is not difficult to check that the latter operator equals 
$\begin{bmatrix} 
v_1 & \theta T \\ 
\ovl{\theta}T^\circ & v_2
\end{bmatrix} $. Consequently the map $|\lambda| 
\cdot 
\begin{bmatrix} 
v_1 & \theta T \\ 
\ovl{\theta}T^\circ & v_2
\end{bmatrix} 
=\begin{bmatrix} 
|\lambda| v_1 & \lambda T \\ 
(\lambda T)^\circ & |\lambda| v_2
\end{bmatrix}$ is also $n$-positive. We deduce that $\lambda T$ is $n$-pseudodecomposable and that $\norm{\lambda T}_{n-\pdec} \leq \max\big\{ \norm{|\lambda| v_1}, \norm{|\lambda| v_2} \big\} = |\lambda| \max\big\{ \norm{v_1}, \norm{v_2} \big\}$. Passing to the infimum yields the desired inequality $\norm{\lambda T}_{n-\pdec} \leq |\lambda| \norm{T}_{n-\pdec}$.
\end{proof}

\begin{prop}
\label{Prop-dec-inf-atteint}
Let $\cal{M}$ and $\cal{N}$ be two von Neumann algebras. Suppose $1 < p < \infty$. Consider some $n \in \{1,2,\ldots,\infty\}$. Let $T \co \L^p(\cal{M}) \to \L^p(\cal{N})$ be an $n$-pseudo-decomposable map. Then the infimum in the definition of $\norm{T}_{n-\pdec}$ is actually a minimum i.e. we can choose $v_1$ and $v_2$ in (\ref{Norm-dec}) such that 
$$
\norm{T}_{n-\pdec,\L^p(M) \to \L^p(N)}
=\max\{\norm{v_1},\norm{v_2}\}.
$$ 
\end{prop}

\begin{proof}
For any integer $n$, let $v_k, w_k \co \L^p(\cal{M}) \to \L^p(\cal{N})$ be bounded maps such that the map $
\begin{bmatrix} 
v_k & T \\ 
T^\circ & w_k
\end{bmatrix} 
\co S^p_2(\L^p(\cal{M})) \to S^p_2(\L^p(N))
$ 
is $n$-positive with $\max\{\norm{v_k},\norm{w_k}\} \leq \norm{T}_{n-\pdec} + \frac{1}{k}$. Note that since the Banach space $\L^p(\cal{N})$ is reflexive, the closed unit ball of the space $\B(\L^p(\cal{M}),\L^p(\cal{N}))$ of bounded operators in the weak operator topology is compact by \cite[Exercise 3.7 (iii) p.~122]{All11}. Hence the bounded sequences $(v_{k})$ and $(w_k)$ admit convergent subnets $(v_{\alpha})$ and $(w_{\alpha})$ in the weak operator topology which converge to some $v,w \in \B(\L^p(\cal{M}),\L^p(\cal{N}))$. Now, it is easy to see that 
$ 
\begin{bmatrix} 
v & T \\ 
T^\circ & w 
\end{bmatrix} 
= \lim_\alpha 
\begin{bmatrix} 
v_{\alpha} & T \\ 
T^\circ & w_{\alpha}
\end{bmatrix}
$ 
in the weak operator topology of $\B(S_2^p(\L^p(\cal{M})),S_2^p(\L^p(\cal{N})))$. By \cite[Lemma 2.8]{ArK23}, the operator on the left hand side is $n$-positive as a weak limit of $n$-positive operators. Moreover, using the weak lower semicontinuity of the norm, we see that $\norm{v} \leq \liminf_\alpha \norm{v_{\alpha}} \leq \norm{T}_{n-\pdec}$ and $\norm{w} \leq \liminf_\alpha \norm{w_{\alpha}} \leq \norm{T}_{n-\pdec}$. Hence, we have $\max\{\norm{v},\norm{w}\}=\norm{T}_{n-\pdec}$. 
\end{proof}

\begin{prop}
\label{Prop-norm-dec}
Let $\cal{M}$ and $\cal{N}$ be two von Neumann algebras. Suppose $1 \leq p \leq \infty$. Consider some $n \in \{1,2,\ldots,\infty\}$. Then the set $\PDec_n(\L^p(\cal{M}),\L^p(\cal{N}))$ of $n$-pseudo-decomposable operators is a vector space and $\norm{\cdot}_{n-\pdec,\L^p(\cal{M}) \to \L^p(\cal{N})}$ is a norm on $\PDec_n(\L^p(\cal{M}),\L^p(\cal{N}))$.
\end{prop}

\begin{proof}
Let $T_1,T_2 \co \L^p(\cal{M}) \to \L^p(\cal{N})$ be $n$-pseudodecomposable maps. There exist some bounded linear maps $v_1,v_2,w_1,w_2 \co \L^p(\cal{M}) \to \L^p(\cal{N})$ such that $\begin{bmatrix} 
   v_1  &  T_1 \\
   T_1^\circ  &  v_2  \\
\end{bmatrix}$ and $\begin{bmatrix} 
    w_1  &   T_2 \\
    T_2^\circ  &  w_2 \\
\end{bmatrix}$ are $n$-positive. We can write 
$
\begin{bmatrix} 
   v_1  &  T_1 \\
   T_1^\circ  &  v_2  \\
\end{bmatrix}
+
\begin{bmatrix} 
    w_1  &   T_2 \\
    T_2^\circ  &  w_2 \\
\end{bmatrix}
=
\begin{bmatrix} 
   v_1+ w_1  &  T_1+T_2 \\
   T_1^\circ + T_2^\circ  &   v_2 + w_2  \\
\end{bmatrix}
=
\begin{bmatrix} 
   v_1+w_1              &  T_1 + T_2 \\
   (T_1+T_2)^\circ  &   v_2+w_2  \\
\end{bmatrix}
$. 
Moreover, this map is $n$-positive. Hence $T_1+T_2$ is $n$-pseudodecomposable. Furthermore, we deduce that
\begin{align*}
\MoveEqLeft
  \norm{T_1+T_2}_{n-\pdec}  
		\leq \max\big\{\norm{v_1+w_1},\norm{v_2+w_2}\big\} \\
		&\leq \max\big\{\norm{v_1}+\norm{w_1}, \norm{v_2}+\norm{w_2}\big\}
		\leq \max\big\{\norm{v_1},\norm{v_2}\big\}+ \max\big\{\norm{w_1},\norm{w_2}\big\}.
\end{align*}
Passing to the infimum, we conclude that the sum $T_1+T_2$ is $n$-pseudo-decomposable and that $\label{prop-dec-triangle-inequality} \norm{T_1+T_2}_{n-\pdec} \leq \norm{T_1}_{n-\pdec}+\norm{T_2}_{n-\pdec}$. The absolute homogeneity is Proposition \ref{prop-dec-homogeneous}. Suppose $\norm{T}_{n-\pdec}=0$. By Proposition \ref{Prop-dec-inf-atteint}, the map $
\begin{bmatrix} 
0 & T \\ 
T^\circ & 0
\end{bmatrix} 
\co S^p_2(\L^p(\cal{M})) \to S^p_2(\L^p(\cal{N}))
$ is $n$-positive. Now, let $x \in \L^p(\cal{M})$. By Lemma \ref{Lemma-Han}, the element
$
\begin{bmatrix} 
|x^*|& x \\ 
x^* & |x|
\end{bmatrix} 
$
of $S^p_{2}(\L^p(\cal{M}))$ is positive. We deduce that the element $
\begin{bmatrix}
0   & T(x)\\
T(x)^* & 0
\end{bmatrix}$ 
is also positive. Using Lemma \ref{Lemma-Matricial-inequality}, we infer that $T(x)=0$. We conclude that $T=0$.
\end{proof}

Now, we give an example.

\begin{prop}
\label{quest-cp-versus-dec1}
Let $\cal{M}$ and $\cal{N}$ be two von Neumann algebras. Suppose $1 \leq p \leq \infty$. Consider some $n \in \{2,\ldots,\infty\}$. Let $T \co \L^p(\cal{M}) \to \L^p(\cal{N})$ be an $n$-positive map. Then $T$ is $\left \lfloor{\frac{n}{2}}\right \rfloor$-pseudo-decomposable and 
$$
\norm{T}_{n-\pdec,\L^p(\cal{M}) \to \L^p(\cal{N})} 
\leq \norm{T}_{\L^p(\cal{M}) \to \L^p(\cal{N})}.
$$ 
\end{prop}

\begin{proof}
We see that the map $\begin{bmatrix} 
T & T \\ 
T & T
\end{bmatrix} 
\co S^p_2(\L^p(\cal{M})) \to S^p_2(\L^p(\cal{N}))$ is $\left \lfloor{\frac{n}{2}}\right \rfloor$-positive. We infer that the map $T$ is $\left \lfloor{\frac{n}{2}}\right \rfloor$-pseudo-decomposable and that the inequality is true.
\end{proof}

\begin{prop} 
Let $\cal{M}$ and $\cal{N}$ be two von Neumann algebras. Suppose $1 \leq p \leq \infty$. Consider some $n \in \{1,\ldots,\infty\}$. An $n$-pseudo-decomposable $T \co \L^p(\cal{M}) \to \L^p(\cal{N})$ is a linear combination of $n$-positive maps.
\end{prop}

\begin{proof}
Suppose that the map $T$ is $n$-pseudo-decomposable. There exist some linear maps $v_1,v_2 \co \L^p(\cal{M}) \to \L^p(\cal{N})$ such that
$\Phi=\begin{bmatrix}
v_1 & T \\ 
T^{\circ} & v_2
\end{bmatrix}$ is $n$-positive. By \eqref{conj-cp}, the maps 
$
T_1=\frac{1}{4}\begin{bmatrix}
1 & 1 \\ 
\end{bmatrix}
\Phi
\begin{bmatrix}
1  \\
1\\ 
\end{bmatrix}
$,
$
T_2=\frac{1}{4}\begin{bmatrix}
1 & -1 \\
\end{bmatrix}
\Phi
\begin{bmatrix}
1  \\
-1\\ 
\end{bmatrix}
$,
$
T_3=\frac{1}{4}\begin{bmatrix}
1 & \i \\ 
\end{bmatrix}
\Phi
\begin{bmatrix}
1  \\
-\i\\ 
\end{bmatrix}
$
and
$
T_4=\frac{1}{4}\begin{bmatrix}
1 & -\i \\ 
\end{bmatrix}
\Phi
\begin{bmatrix}
1  \\
\i\\ 
\end{bmatrix}
$ are $n$-positive from $\L^p(\cal{M})$ into $\L^p(\cal{N})$ and it is easy to check that $T=T_1-T_2+\i(T_3-T_4)$.
\end{proof}

This result allows us to describe the $n$-pseudo-decomposable maps on classical $\L^p$-spaces.

\begin{example}\normalfont
Let $\Omega$ and $\Omega'$ be $\sigma$-finite measure spaces. If $n \geq 1$, any $n$-pseudo-decomposable map $T \co \L^p(\Omega) \to \L^p(\Omega)$ is necessarily  regular (=decomposable by \cite[Theorem 3.24]{ArK23}). The converse is obviously true.
\end{example}

\begin{remark} \normalfont
A similar notion of $n$-pseudo-decomposable map could be defined between ternary rings of operators, generalizing the decomposable maps of \cite[Section 7]{KaR02}. 
\end{remark}

\section{Complementation of rectangular $\L^p$-spaces in noncommutative $\L^p$-spaces}
\label{Sec-complementation}

We start by giving examples of decomposable maps on noncommutative $\L^p$-spaces. We will also use this observation in the proof of Theorem \ref{main-th-ds-intro}. In the two following results, we can use Haagerup noncommutative $\L^p$-spaces or Kosaki noncommutative $\L^p$-spaces \cite{Kos84} (see also \cite{Ray03} for the construction of the positive cone).

\begin{lemma}
\label{Mult-decomposable}
Let $\cal{M}$ be a von Neumann algebra equipped with a normal semifinite faithful weight $\varphi$. Suppose $1 \leq p \leq \infty$. Let $a,b \in \cal{M}$. Then the two-sided multiplication operator $M_{a,b} \co \L^p(\cal{M}) \to \L^{p}(\cal{M})$, $x \mapsto axb$ is decomposable and we have
\begin{equation}
\label{estimation-dec-norm}
\norm{M_{a,b}}_{\dec, \L^p(\cal{M}) \to \L^{p}(\cal{M})}
\leq \norm{a}_{\L^\infty(\cal{M})} \norm{b}_{\L^\infty(\cal{M})}.
\end{equation}
More precisely, the map $\Phi \ov{\mathrm{def}}{=} \begin{bmatrix}
   M_{a,a^*}  &  M_{a,b} \\
   M_{a,b}^\circ  &  M_{b^*,b}  \\
\end{bmatrix} \co S^p_2(\L^p(\cal{M})) \to S^p_2(\L^p(\cal{M}))$ is completely positive.
\end{lemma}

\begin{proof}
For any $z \in \cal{M}$, we see that $M_{a,b}^\circ(z)=(M_{a,b}(z^*))^*=(a z^* b)^*=b^*za^*$. 
Consequently, for any element $\begin{bmatrix}
   x  &  y \\
   z &  t  \\
\end{bmatrix}$ of the space $S^{p}_2(\L^p(\cal{M}))$, we have
\begin{equation*}
\label{}
\Phi\left(\begin{bmatrix}
   x  &  y \\
   z &  t  \\
\end{bmatrix}\right)
=\begin{bmatrix}
   a x a^*  &  ayb \\
   b^*za^*  &  b^*tb  \\
\end{bmatrix}
=\begin{bmatrix}
   a  &  0 \\
   0  &  b^*  \\
\end{bmatrix}
\begin{bmatrix}
   x  &  y \\
   z  &  t  \\
\end{bmatrix}
\begin{bmatrix}
   a^*  &  0 \\
   0  &  b  \\
\end{bmatrix}
=\begin{bmatrix}
   a  &  0 \\
   0  &  b^*  \\
\end{bmatrix}
\begin{bmatrix}
   x  &  y \\
   z  &  t  \\
\end{bmatrix}
\begin{bmatrix}
   a  &  0 \\
   0  &  b^* \\
\end{bmatrix}^*.
\end{equation*}
We deduce that $\Phi$ is completely positive. By \eqref{Norm-dec-intro}, we infer that $\norm{M_{a,b}}_{\dec,\L^p(\cal{M}) \to \L^p(\cal{M})} \leq \max\{\norm{M_{a,a^*}},\norm{M_{b^*,b}}\}$. By H\"older's inequality, we see that for any $x \in \L^p(\cal{M})$ we have $\norm{M_{a,a^*}(x)}_{\L^p(\cal{M})}=\norm{a x a^*}_{\L^p(\cal{M})} \leq \norm{x}_{\L^p(\cal{M})}\norm{a}_{\L^\infty(\cal{M})}^2$. So $\norm{M_{a,a^*}} \leq \norm{a}_{\L^\infty(\cal{M})}^2$ and similarly $\norm{M_{b^*,b}} \leq \norm{b}_{\L^\infty(\cal{M})}^2$. We obtain that 
$$
\norm{M_{a,b}}_{\dec,\L^p(\cal{M}) \to \L^{p}(\cal{M})}
\leq \max\big\{\norm{a}_{\L^\infty(\cal{M})}^2,\norm{b}_{\L^\infty(\cal{M})}^2\big\}.
$$
Replacing $a$ by $\frac{1}{\norm{a}}a$ and $b$ by $\frac{1}{\norm{b}} b$, we conclude that\footnote{\thefootnote. Indeed, we have 
$$
\frac{1}{\norm{a}_\infty\norm{b}_\infty}\norm{M_{a,b}}_{\dec,\L^p(\cal{M}) \to \L^p(\cal{M})}
=\bgnorm{M_{\frac{1}{\norm{a}}a,\frac{1}{\norm{b}}b}}_{\dec,\L^p(\cal{M}) \to \L^p(\cal{M})}
\leq \max\bigg\{\bgnorm{\frac{1}{\norm{a}}a}^2,\bgnorm{\frac{1}{\norm{b}}b}^2 \bigg\}.
$$} the inequality \eqref{estimation-dec-norm} is true.
\end{proof}

\begin{prop}
\label{Dec-proj}
Let $e,f$ be some orthogonal projections of a von Neumann algebra $\cal{M}$ equipped with a normal semifinite faithful weight $\varphi$. Suppose $1 \leq p \leq \infty$. The linear map $P \co \L^p(\cal{M}) \to \L^p(\cal{M})$, $x \mapsto exf$ is a contractively decomposable projection whose range is the closed subspace $e\L^p(\cal{M})f$.
\end{prop}

\begin{proof}
By Lemma \ref{Mult-decomposable}, the map $P$ is contractively decomposable. For any $x \in \L^p(\cal{M})$, we have $P^2(x)=P(exf)=e^2xf^2=exf=P(x)$. So $P^2=P$, i.e. $P$ is a projection. The last assertion is clear.
\end{proof}

\begin{remark} \normalfont
Recall that a linear projection on a $\JBW^*$-triple $X$ is said structural if for any elements $x,y,z$ of $X$ we have 
$$
P(\{x, P(y), z\})
=\{P(x), y, P(z)\}.
$$
By \cite[Theorem 5.3]{ECR96}, such a projection is necessarily contractive and weak* continuous. Moreover, it is proved in \cite{ECR96} that the map $P \mapsto P(X)$ is a bijection from the set of structural projections on $X$ onto the set of weak* closed inner ideals of $X$ (i.e. the set of weak* closed subspaces $J$ of $X$ satisfying $\{J,X,J\} \subset J$). If $p=\infty$, note that by \cite[Theorem 6.1]{ECR96} the projections $P\co \L^\infty(\cal{M}) \to \L^\infty(\cal{M})$ of Proposition \ref{Dec-proj} are \textit{exactly} the structural projections of $\cal{M}$. We refer to \cite{EHR03} for another characterization of structural projections on $\JBW^*$-triples and to \cite{EdR89} and \cite{EdR98b} for more information on weak* closed inner ideals.
\end{remark}

In the spirit of \cite[p.~869]{KaR02}, we define the rectangular $\L^p$-spaces of $\W^*$-$\TRO$s using Kosaki noncommutative $\L^p$-spaces of \cite{Kos84} \cite{Ray03}. Indeed, here we need an inclusion of the linking von Neumann algebra in its noncommutative $\L^p$-spaces. So we cannot use Haagerup noncommutative $\L^p$-spaces in this definition.

\begin{defi}
Let $V$ be a $\W^*$-$\TRO$ with linking von Neumann algebra $R(V)$. Suppose  that the von Neumann algebra $R(V)$ is $\sigma$-finite equipped with a normal faithful state $\varphi$. Suppose $1 \leq p < \infty$. We define the rectangular $\L^p$-space $\L^p(V,\varphi)$ to be the norm closure of $V=eR(V)e^\perp$ in the Kosaki noncommutative $\L^p$-space $\L^p(R(V),\varphi)$. 
\end{defi}

It is easy to check that $\L^p(V,\varphi)=e\L^p(R(V),\varphi)e^\perp$ since $R(V)$ is dense in the Banach space $\L^p(R(V),\varphi)$. We let $\L^\infty(V,\varphi) \ov{\mathrm{def}}{=} V$. An immediate use of Proposition \ref{Dec-proj} gives the following result.

\begin{cor}
Let $V$ be a $\W^*$-$\TRO$ such that the linking von Neumann algebra $R(V)$ is $\sigma$-finite and equipped with a normal faithful state $\varphi$. Suppose $1 \leq p \leq \infty$. Then the rectangular $\L^p$-space $\L^p(V,\varphi)$ is a contractively complemented subspace of the Kosaki noncommutative $\L^p$-space $\L^p(R(V),\varphi)$ by a contractively decomposable projection.
\end{cor}

\begin{remark} \normalfont
Recall that Kosaki noncommutative $\L^p$-spaces are defined by interpolation. Using Lemma \ref{Lemma-interpolation} combined with Proposition \ref{Dec-proj}, it is immediate that we have the complex interpolation formula $\L^p(V,\varphi)=(\L^\infty(V,\varphi),\L^1(V,\varphi))_{\frac{1}{p}}$. 

The paper \cite[Appendix 6]{GJL18} contains related results. However, note that in \cite[Appendix 6]{GJL18}, it is written that for any $\W^*$-$\TRO$ $V$ of $\B(H)$ we have an isometry $V_p=(V,V_1)_{\frac{1}{p}}$ where $V_p$ is the closure of the intersection $X \cap S^p(H)$ in the Schatten space $S^p(H)$,  relying on the existence of a couple of \textit{compatible} contractive projections $(P_\infty, P_1)$ on $V$ and $V_1$ (and Lemma \ref{Lemma-interpolation}). However, this existence seems unclear in the general case where $V$ is not a weak* closed inner ideal of $\B(H)$. The concrete representation of $V$ in $\B(H)$ seems to be crucial.
\end{remark}

\begin{remark} \normalfont
Using the construction of noncommutative $\L^p$-spaces of \cite{Ter82}, it is left to the interested reader to generalize the previous corollary to arbitrary $\W^*$-TROs using weights instead states. In this situation, a von Neumann algebra is not a subset of its noncommutative $\L^p$-spaces.
\end{remark}

\section{A lifting for $n$-pseudo-decomposable maps on noncommutative $\L^p$-spaces} 
\label{Sec-lifting}

Our main tool will be the following result \cite[Theorem 2.7]{ArR19} which is a lifting theorem for positive maps acting on noncommutative $\L^p$-spaces which has its roots in \cite[Theorem 3.1]{JRX05}. 

\begin{thm}
\label{Th-relevement-cp} 
Let $\cal{M}$ and $\cal{N}$ be von Neumann algebras and $n \in \{1,2,\ldots,\infty\}$. Suppose $1 \leq p < \infty$. Let $T \co \L^p(\cal{M}) \to \L^p(\cal{N})$ be an $n$-positive linear map. Let $h$ be a positive element of $\L^p(\cal{M})$. 
Then there exists a unique linear map $v \co \cal{M} \to s(T(h))\cal{N}s(T(h))$ such that 
\begin{equation}
\label{equa-relevement}
T\big(h^{\frac{1}{2}}xh^{\frac{1}{2}}\big)
=T(h)^{\frac{1}{2}}v(x)T(h)^{\frac{1}{2}},\qquad x \in \cal{M}.
\end{equation}
Moreover, this map $v$ is unital, $n$-positive, contractive and normal.
\end{thm}

Now, in the same spirit we prove a lifting theorem for $n$-pseudo-decomposable maps.

\begin{thm}
\label{Lifting-decomposable}
Let $\cal{M}$ and $\cal{N}$ be von Neumann algebras and $n \in \{1,2,\ldots,\infty\}$. Suppose $1 \leq p < \infty$. Let $T \co \L^p(\cal{M}) \to \L^p(\cal{N})$ be an $n$-pseudo-decomposable map and $v_1,v_2 \co \L^p(\cal{M}) \to \L^p(\cal{N})$ be some maps such that the operator $\Phi$ of \eqref{Matrice-2-2-Phi} is $n$-positive. Let $h$ and $k$ be positive elements of $\L^p(\cal{M})$. Then there exists a unique linear map $w \co \cal{M} \to s(v_1(h))\cal{N}s(v_2(k))$ such that
\begin{equation}
\label{lifting-decomposable}
T\big(h^{\frac{1}{2}}xk^{\frac{1}{2}}\big)
=v_1(h)^{\frac{1}{2}}w(x)v_2(k)^{\frac{1}{2}}
\end{equation}
for any $x \in \cal{M}$. Moreover, $w$ is normal and contractively $n$-pseudo-decomposable. 
\end{thm}

\begin{proof}
Consider the positive element $H \ov{\mathrm{def}}{=} \begin{bmatrix}
  h   &  0 \\
  0  &  k  \\
\end{bmatrix}$ of the space $S^p_2(\L^p(\cal{M}))$. The support of the element $\Phi(H)$ of $S^p_2(\L^p(\cal{M}))$ is given by 
\begin{equation}
\label{Support}
s(\Phi(H))
=s\left(\Phi\left(\begin{bmatrix}
  h   &  0 \\
  0  &  k  \\
\end{bmatrix}\right)\right)
\ov{\eqref{Matrice-2-2-Phi}}{=}s\left(\begin{bmatrix}
  v_1(h)   &  0 \\
  0  &  v_2(k)  \\
\end{bmatrix}\right)
=\begin{bmatrix}
  s(v_1(h))   &  0 \\
  0  &  s(v_2(k))  \\
\end{bmatrix}.
\end{equation}
Using Theorem \ref{Th-relevement-cp} with the $n$-positive operator \eqref{Matrice-2-2-Phi} and the positive element $H$ instead of $h$, we see that there exists a unique linear map $v \co \M_2(\cal{M}) \to s(\Phi(H))\M_2(\cal{N})s(\Phi(H))$ such that for any element $X\ov{\mathrm{def}}{=} \begin{bmatrix}
  x_{11}   &  x_{12} \\
  x_{21}  & x_{22}   \\
\end{bmatrix}$ of $\M_2(\cal{M})$ we have
\begin{equation}
\label{Equa-divers-3421}
\Phi\left(\begin{bmatrix}
  h   &  0 \\
  0  &  k  \\
\end{bmatrix}^{\frac{1}{2}}
\begin{bmatrix}
  x_{11}   &  x_{12} \\
  x_{21}  & x_{22}   \\
\end{bmatrix}\begin{bmatrix}
  h   &  0 \\
  0  &  k  \\
\end{bmatrix}^{\frac{1}{2}}\right)
=\Phi\left(\begin{bmatrix}
  h   &  0 \\
  0  &  k  \\
\end{bmatrix}\right)^{\frac{1}{2}}
v(X)
\Phi\left(\begin{bmatrix}
  h   &  0 \\
  0  &  k  \\
\end{bmatrix}\right)^{\frac{1}{2}}.
\end{equation}
Moreover, the map $v$ is a normal $n$-positive contraction. On the one hand, we have
\begin{align}
\label{Divers-457}
 \Phi\left(\begin{bmatrix}
  h   &  0 \\
  0  &  k  \\
\end{bmatrix}^{\frac{1}{2}}
\begin{bmatrix}
  x_{11}   &  x_{12} \\
  x_{21}  & x_{22}   \\
\end{bmatrix}\begin{bmatrix}
  h   &  0 \\
  0  &  k  \\
\end{bmatrix}^{\frac{1}{2}}\right)           
&=\Phi\left(\begin{bmatrix}
  h^{\frac{1}{2}} x_{11} h^{\frac{1}{2}}  & h^{\frac{1}{2}}  x_{12}k^{\frac{1}{2}}  \\
 k^{\frac{1}{2}}  x_{21} h^{\frac{1}{2}}  & k^{\frac{1}{2}}  x_{22} k^{\frac{1}{2}}   \\
\end{bmatrix}\right)\\ 
&\ov{\eqref{Matrice-2-2-Phi}}{=} \begin{bmatrix}
  v_1(h^{\frac{1}{2}} x_{11} h^{\frac{1}{2}})   & T(h^{\frac{1}{2}}  x_{12}k^{\frac{1}{2}})  \\
 T^\circ(k^{\frac{1}{2}} x_{21} h^{\frac{1}{2}})  & v_2(k^{\frac{1}{2}}  x_{22} k^{\frac{1}{2}})   \\
\end{bmatrix} \nonumber.
\end{align} 
On the other hand, we have
\begin{align}
\MoveEqLeft
\label{Divers-456}
 \Phi\left(\begin{bmatrix}
  h   &  0 \\
  0  &  k \\
\end{bmatrix}\right)^{\frac{1}{2}}
v(X)
\Phi\left(\begin{bmatrix}
  h   &  0 \\
  0  &  k \\
\end{bmatrix}\right)^{\frac{1}{2}} 
\ov{\eqref{Matrice-2-2-Phi}}{=} 
\begin{bmatrix}
  v_1(h)   &  0 \\
  0  &  v_2(k)  \\
\end{bmatrix}^{\frac{1}{2}}
v(X)
\begin{bmatrix}
  v_1(h)   &  0 \\
  0  &  v_2(k)  \\
\end{bmatrix}^{\frac{1}{2}} \\
&=\begin{bmatrix}
  v_1(h)^{\frac{1}{2}}   &  0 \\
  0  &  v_2(k)^{\frac{1}{2}}  \\
\end{bmatrix}
v(X)
\begin{bmatrix}
  v_1(h)^{\frac{1}{2}}   &  0 \\
  0  &  v_2(k)^{\frac{1}{2}}  \\
\end{bmatrix}. \nonumber 
\end{align} 
Consequently, for any element $X
=\begin{bmatrix}
  x_{11}  &  x_{12} \\
  x_{21}  & x_{22}   \\
\end{bmatrix}$ of $\M_2(\cal{M})$, we have
\begin{equation}
\label{Equa-fond-1234}
\begin{bmatrix}
  v_1(h^{\frac{1}{2}} x_{11} h^{\frac{1}{2}}) & T(h^{\frac{1}{2}} x_{12}k^{\frac{1}{2}})  \\
 T^\circ(k^{\frac{1}{2}} x_{21} h^{\frac{1}{2}})  & v_2(k^{\frac{1}{2}} x_{22} k^{\frac{1}{2}})   \\
\end{bmatrix}
\ov{\eqref{Equa-divers-3421} \eqref{Divers-457} \eqref{Divers-456}}{=} \begin{bmatrix}
  v_1(h)^{\frac{1}{2}}   &  0 \\
  0  &  v_2(k)^{\frac{1}{2}}  \\
\end{bmatrix}
v(X)
\begin{bmatrix}
  v_1(h)^{\frac{1}{2}}   &  0 \\
  0  &  v_2(k)^{\frac{1}{2}}  \\
\end{bmatrix}.
\end{equation}
Now, we introduce the map 
$$
u \co \cal{M} \to s(\Phi(H))\M_2(\cal{N})s(\Phi(H))\ov{\eqref{Support}}{=}\begin{bmatrix}
  s(v_1(h))\cal{N}s(v_1(h))   &  s(v_1(h))\cal{N}s(v_2(k)) \\
  s(v_2(k))\cal{N}s(v_1(h))  &  s(v_2(k))\cal{N}s(v_2(k))  \\
\end{bmatrix}
$$ 
defined by
$u(x)
\ov{\mathrm{def}}{=} v\left(\begin{bmatrix}
  x   &  x \\
  x  &  x  \\
\end{bmatrix}\right)$. The map $\cal{M} \mapsto \M_2(\cal{M})$,  $x \mapsto \begin{bmatrix}
  x   &  x \\
  x  &  x  \\
\end{bmatrix}=
\begin{bmatrix}
1  \\
1\\ 
\end{bmatrix} 
x
\begin{bmatrix}
1  &1 \\ 
\end{bmatrix}$ is completely positive and obviously normal. By composition, the map $u$ is normal and $n$-positive. We can write $u=\begin{bmatrix}
  u_{1}   &  w\\
  u_3  &  u_{2}  \\
\end{bmatrix}$ for some normal maps $u_1 \co \cal{M} \to s(v_1(h))\cal{N}s(v_1(h))$, $u_2 \co \cal{M} \to s(v_2(k))\cal{N}s(v_2(k))$, $u_3 \co \cal{M} \to s(v_2(k))\cal{N}s(v_1(h))$ and $w \co \cal{M} \to s(v_1(h))\cal{N}s(v_2(k))$. The map $u$ is positive hence $*$-preserving. Thus, for any $x \in \cal{M}$ we have 
$$
\begin{bmatrix}
  u_{1}(x)^*   & u_{3}(x)^*  \\
  w(x)^*  &  u_{2}(x)^*  \\
\end{bmatrix}=\begin{bmatrix}
  u_{1}(x)   &  w(x) \\
  u_{3}(x)  &  u_{2}(x)  \\
\end{bmatrix}^*
=u(x)^*
=u(x^*)
=\begin{bmatrix}
  u_{1}(x^*)   &  w(x^*) \\
  u_{3}(x^*)  &  u_{2}(x^*)  \\
\end{bmatrix}.
$$ 
In particular, we deduce that $w(x^*)=u_{3}(x)^*$ for any $x \in \cal{M}$. Hence $u_{3}=w^\circ$. So $u=\begin{bmatrix}
  u_{1}   &  w\\
  w^\circ  &  u_{2}  \\
\end{bmatrix}$. This implies that $w$ is $n$-pseudo-decomposable and that $u_1$ and $u_2$ are $n$-positive. Moreover, for any $x \in \cal{M}$, we see that 
\begin{align*}
\begin{bmatrix}
  v_1(h^{\frac{1}{2}} x h^{\frac{1}{2}}) & T(h^{\frac{1}{2}} xk^{\frac{1}{2}})  \\
 T^\circ(k^{\frac{1}{2}}  x h^{\frac{1}{2}})  & v_2(k^{\frac{1}{2}} x k^{\frac{1}{2}})   \\
\end{bmatrix}
&\ov{\eqref{Equa-fond-1234}}{=} \begin{bmatrix}
  v_1(h)^{\frac{1}{2}}   &  0 \\
  0  &  v_2(k)^{\frac{1}{2}}  \\
\end{bmatrix}
v\left(\begin{bmatrix}
  x   &  x \\
  x  &  x  \\
\end{bmatrix}\right)
\begin{bmatrix}
  v_1(h)^{\frac{1}{2}}   &  0 \\
  0  &  v_2(k)^{\frac{1}{2}}  \\
\end{bmatrix} \\
&=\begin{bmatrix}
  v_1(h)^{\frac{1}{2}}   &  0 \\
  0  &  v_2(k)^{\frac{1}{2}}  \\
\end{bmatrix}
\begin{bmatrix}
		u_1(x)& w(x)  \\
  w^\circ(x)  & u_2(x)  \\
\end{bmatrix}
\begin{bmatrix}
v_1(h)^{\frac{1}{2}}   &  0 \\
0  & v_2(k)^{\frac{1}{2}}  \\
\end{bmatrix}
\\
&=\begin{bmatrix}
v_1(h)^{\frac{1}{2}}u_1(x)v_1(h)^{\frac{1}{2}}& v_1(h)^{\frac{1}{2}}w(x)v_2(k)^{\frac{1}{2}} \\
v_2(k)^{\frac{1}{2}}w^\circ(x)v_1(h)^{\frac{1}{2}}  & v_2(k)^{\frac{1}{2}}u_2(x)v_2(k)^{\frac{1}{2}} \\
\end{bmatrix}.
\end{align*}
Looking at the (1-2)-entry, we obtain \eqref{lifting-decomposable}. Looking at the other entries, we deduce by uniqueness that $u_{1}$ and $u_{2}$ are the liftings given by Theorem \ref{Th-relevement-cp} of the $n$-positive maps $v_1,v_2 \co \L^p(\cal{M}) \to \L^p(\cal{N})$. Consequently, we have the inequalities $\norm{u_1} \leq 1$ and $\norm{u_2} \leq 1$ and thus $\norm{w}_{n-\pdec, \cal{M} \to \cal{N}} \ov{\eqref{Norm-dec}}{\leq} 1$. The uniqueness is obvious. 
\end{proof}

The case $n=\infty$ gives the following particular case.

\begin{cor}
\label{CorLifting-decomposable}
Let $\cal{M}$ and $\cal{N}$ be von Neumann algebras. Suppose $1 \leq p < \infty$. Let $T \co \L^p(\cal{M}) \to \L^p(\cal{N})$ be a decomposable map and $v_1,v_2 \co \L^p(\cal{M}) \to \L^p(\cal{N})$ be some maps such that the operator \eqref{Matrice-2-2-Phi} is completely positive. Let $h$ and $k$ be positive elements of $\L^p(\cal{M})$. Then there exists a unique linear map  $w \co \cal{M} \to s(v_1(h))\cal{N}s(v_2(k))$ such that
\begin{equation}
\label{lifting-decomposable-2}
T\big(h^{\frac{1}{2}}xk^{\frac{1}{2}}\big)
=v_1(h)^{\frac{1}{2}}w(x)v_2(k)^{\frac{1}{2}}
\end{equation}
for any $x \in \cal{M}$. Moreover, $w$ is normal and contractively decomposable.
\end{cor}

\begin{remark} \normalfont
Let $\cal{M}$ be a von Neumann algebra. By \cite[Theorem 1.6]{Haa85}, if the von Neumann algebra $\cal{N}$ is injective, we have a canonical isometry $\CB(\cal{M},\cal{N})=\Dec(\cal{M},\cal{N})$, that means that any completely bounded map $T \co \cal{M} \to \cal{N}$ is necessary decomposable with $\norm{T}_{\cb,\cal{M} \to \cal{N}}=\norm{T}_{\dec,\cal{M} \to \cal{N}}$. 
\end{remark}

\section{Contractively $n$-pseudo-decomposable projections} 
\label{Sec-Contractively-decomposable-projections}

Let $X$ be a Banach space and $T \co X \to X$ be a bounded operator. For any integer $m \geq 1$, we define the average 
$
A_{m,T}
\ov{\mathrm{def}}{=} \frac{1}{m} \sum_{k=1}^{m} T^k
$
of the first $m$ iterates of $T$. Now, we use ergodic theory to obtain information on contractively $n$-pseudo-decomposable projections on noncommutative $\L^p$-spaces. It is \textit{important} to note in the following result that the map $\begin{bmatrix}
   P_1  &  P \\
   P^\circ  &  P_2  \\
\end{bmatrix}$ is (a priori) not necessarily contractive. It is related to Question \ref{Open-quest}. 

\begin{prop}
\label{prop-cd-proj-v1v2-proj}
Let $\cal{M}$ be a von Neumann algebra and $n \in \{1,2,\ldots,\infty\}$. Suppose $1<p<\infty$. Let $P \co \L^p(\cal{M}) \to \L^p(\cal{M})$ be a contractively $n$-pseudo-decomposable projection. There exist  ($n$-positive) contractive projections $P_1,P_2 \co \L^p(\cal{M}) \to \L^p(\cal{M})$ such that the linear map $
\begin{bmatrix}
   P_1  &  P \\
   P^\circ  &  P_2  \\
\end{bmatrix}
\co S^p_2(\L^p(\cal{M})) \to S^p_2(\L^p(\cal{M}))$ is an $n$-positive projection.
\end{prop}

\begin{proof}
By Proposition \ref{Prop-dec-inf-atteint}, there exist bounded linear maps $v_1,v_2 \co \L^p(\cal{M}) \to \L^p(\cal{M})$ such that the linear map
$
\Phi
=\begin{bmatrix}
   v_1  &  P \\
   P^\circ  &  v_2  \\
\end{bmatrix}
\co S^p_2(\L^p(\cal{M})) \to S^p_2(\L^p(\cal{M}))$ is $n$-positive with $\max\{\norm{v_1},\norm{v_2}\}=1$. By composition, note that the map $\Phi^k$ is $n$-positive for any integer $k \geq 1$. Since $P^2=P$, we have $
\Phi^k
=\begin{bmatrix}
   v_1  &  P \\
   P^\circ  &  v_2  \\
\end{bmatrix}^k
=\begin{bmatrix}
   v_1^k  &  P^k \\
   (P^k)^\circ  &  v_2^k  \\
\end{bmatrix}
=\begin{bmatrix}
   v_1^k  &  P \\
   P^\circ  &  v_2^k  \\
\end{bmatrix}$ for any $k\geq 1$. For any integer $m \geq 1$, we infer that the average 
$$
A_{m,\Phi}
=\frac{1}{m} \sum_{k=1}^{m} \Phi^k
=\frac{1}{m}\sum_{k=1}^{m}\begin{bmatrix}
   v_1^k  &  P \\
   P^\circ  &  v_2^k  \\
\end{bmatrix}
=\begin{bmatrix}
   \frac{1}{m}\sum_{k=1}^{m}v_1^k  &  P \\
   P^\circ  &  \frac{1}{m}\sum_{k=1}^{m}v_2^k  \\
\end{bmatrix}
=\begin{bmatrix}
   A_{m,v_1}  &  P \\
   P^\circ  &  A_{m,v_2}  \\
\end{bmatrix}
$$ 
is $n$-positive. The maps $v_1$ and $v_2$ are contractions, hence power-bounded. By \cite[Theorem 8.22 p.~149]{EFHN15}\footnote{\thefootnote. In \cite{EFHN15}, the averages are defined with a sum $\sum_{k=0}^{m-1}$. However, it is obvious that the result is also true with a sum $\sum_{k=1}^{m}$.}, since the Banach space $\L^p(\cal{M})$ is reflexive ($1<p<\infty$), we deduce that the maps $v_1$ and $v_2$ are mean ergodic. This means that the sequences $(A_{m,v_1})$ and $(A_{m,v_2})$ converge for the strong operator topology of the space $\B(\L^p(\cal{M}))$ to some bounded operators $P_1,P_2 \co \L^p(\cal{M}) \to \L^p(\cal{M})$. By \cite[Lemma 8.3 p.~137]{EFHN15}, the operators $P_1$ and $P_2$ are projections. Since each average $A_{m,v_i}$ ($i=1,2$) is clearly contractive, by the strong lower semicontinuity of the norm, the maps $P_1$ and $P_2$ are also contractive. It is obvious that the sequence $(A_{m,\Phi})$ of $n$-positive maps converges strongly (hence weakly) to the bounded operator $\begin{bmatrix}
   P_1  &  P \\
   P^\circ  &  P_2  \\
\end{bmatrix}$. By \cite[Lemma 2.10]{ArK23}, we conclude that the map $\begin{bmatrix}
   P_1  &  P \\
   P^\circ  &  P_2  \\
\end{bmatrix}$ is $n$-positive.   
\end{proof}

The particular case $n=\infty$ gives the following.

\begin{cor}
Let $\cal{M}$ be a von Neumann algebra. Suppose $1<p<\infty$. Let $P \co \L^p(\cal{M}) \to \L^p(\cal{M})$ be a contractively decomposable projection. There exist contractive (completely positive) projections $P_1,P_2 \co \L^p(\cal{M}) \to \L^p(\cal{M})$ such that the map $
\begin{bmatrix}
   P_1  &  P \\
   P^\circ  &  P_2  \\
\end{bmatrix}
\co S^p_2(\L^p(\cal{M})) \to S^p_2(\L^p(\cal{M}))$ is a completely positive projection.
\end{cor}

Suppose $1 \leq p < \infty$. Let $\cal{M}$ be a $\sigma$-finite ($=$ countably decomposable) von Neumann algebra and $P \co \L^p(\cal{M}) \to \L^p(\cal{M})$ be a positive contractive projection. In \cite[Section 6]{ArR19}, we introduced the support $s(P)$ of $\Ran P$ as the supremum in $\cal{M}$ of the supports of the positive elements in $\Ran P$:
\begin{equation}
\label{support-s(P)}
s(P)\ov{\mathrm{def}}{=}\bigvee_{h \in \Ran P, h \geq 0} s(h).
\end{equation}
Under these assumptions, by \cite[Proposition 6.1]{ArR19}, there exists a positive element $h$ of $\Ran P$ such that 
\begin{equation}
\label{support-h}
s(P)
=s(h).
\end{equation}
We have $P(h)=h$. In this case, by \cite[Proposition 6.4]{ArR19}, we have
\begin{equation}
\label{Invariance-2}
P(y)
=P\big(s(h)ys(h)\big), \quad y \in \L^p(\cal{M}).
\end{equation}

The following theorem is the main result of this paper.

\begin{thm}
\label{main-th-ds-bis}
Let $\cal{M}$ be a $\sigma$-finite von Neumann algebra equipped with a normal faithful state $\varphi$. Suppose $1 < p<\infty$ and $n \in \{1,2,\ldots,\infty\}$. A bounded map $P \co \L^p(\cal{M}) \to \L^p(\cal{M})$ is a contractively $n$-pseudo-decomposable projection if and only if there exist a normal faithful positive linear form $\phi$ on the von Neumann algebra $\M_2(\cal{M})$, two positive elements $h,k \in \L^p(\cal{M})$ with support projection $s(h)$ and $s(k)$, some linear maps $w \co \cal{M} \to s(h)\cal{M}s(k)$, $u_1 \co \cal{M} \to s(h)\cal{M}s(h)$, $u_2 \co \cal{M} \to s(k)\cal{M}s(k)$ such that the map $Q=\begin{bmatrix}
u_1 & w \\
w^\circ  & u_2 \\
\end{bmatrix}$ and the element $H=\begin{bmatrix}
h & 0 \\
0  & k \\
\end{bmatrix}$ of $S^p_2(\L^p(\cal{M}))$ satisfy the following properties:
\begin{enumerate}
\item for any $x \in \L^p(\cal{M})$ we have $P(x)=P(s(h)xs(k))$,

\item $Q$ is a weak* continuous unital contractive $n$-positive map and its restriction on the von Neumann algebra $s(H)\M_2(\cal{M})s(H)$ is a faithful projection,

\item $s(H)$ belongs to the centralizer of $\phi$ and $\phi_{s(H)}=(\tr \ot \varphi)_{s(H)}$,

\item for any $y \in s(H)\M_2(\cal{M})s(H)$ we have
$
(\tr \ot \tr_{\varphi})(H^p Q(y))
=(\tr \ot \tr_{\varphi})(H^p y)$,


\item for any $x \in s(h)\cal{M}s(k)$, we have 
\begin{equation}
\label{relevement-dec-intro}
\quad P\big(h^{\frac{1}{2}}xk^{\frac{1}{2}}\big)
= h^{\frac{1}{2}} w(x) k^{\frac{1}{2}}, \qquad x \in \cal{M}.
\end{equation}
\end{enumerate}
In this case, the restriction of $w$ is a weak* continuous $n$-pseudodecomposable projection and the range of this map admits a structure of $\JW^*$-triple. Moreover, if $n=\infty$ then the range of $w$ admits a structure of $\W^*$-$\TRO$.
\end{thm}

\begin{proof}
By Proposition \ref{prop-cd-proj-v1v2-proj}, there exist contractive ($n$-positive) projections $P_1,P_2 \co \L^p(\cal{M}) \to \L^p(\cal{M})$ such that the map 
\begin{equation}
\label{Phi-P1}
\Phi \ov{\mathrm{def}}{=}
\begin{bmatrix}
   P_1  &  P \\
   P^\circ  &  P_2  \\
\end{bmatrix}
\co S^p_2(\L^p(\cal{M})) \to S^p_2(\L^p(\cal{M}))
\end{equation}
is an $n$-positive projection. We fix a normal faithful state $\varphi$ on $\cal{M}$ such that $\L^p(\cal{M})=\L^p(\cal{M},\varphi)$. Considering the projection $e \ov{\mathrm{def}}{=} \begin{bmatrix}
   s(P_1)  &  0 \\
   0  &  s(P_2) \\
\end{bmatrix}$, we can define with the construction \eqref{Extension-weight1} the normal faithful positive linear form $\phi$ on the von Neumann algebra $\M_2(\cal{M})$ by 
$$
\phi(x)
\ov{\mathrm{def}}{=}(\tr \ot \varphi)_{e}(exe)+(\tr \ot \varphi)_{e^\perp}(e^\perp xe^\perp), \quad x \in \M_2(\cal{M}).
$$
The projection $e$ belongs to the centralizer of $\phi$ and $\phi_e=(\tr \ot \varphi)_e$. By \eqref{support-h}, there exist positive elements $h$ and $k$ of $\Ran P_1$ and $\Ran P_2$ respectively such that $s(h)=s(P_1)$ and $s(k)=s(P_2)$. 
We let $H \ov{\mathrm{def}}{=}\begin{bmatrix}
   h  &  0 \\
   0  &  k \\
\end{bmatrix}$ and we introduce the algebra
$$
\cal{N}
\ov{\mathrm{def}}{=} s(H)\M_2(\cal{M})s(H)
=\begin{bmatrix}
   s(h)  &  0 \\
   0  &  s(k) \\
\end{bmatrix}\M_2(\cal{M}) \begin{bmatrix}
   s(h)  &  0 \\
   0  &  s(k) \\
\end{bmatrix}
=\begin{bmatrix}
   s(h)\cal{M}s(h)  &  s(h)\cal{M}s(k) \\
   s(k)\cal{M}s(h)  &  s(k)\cal{M}s(k) \\
\end{bmatrix}.
$$

\begin{lemma}
\label{final}
For any $x \in \L^p(\cal{M})$, we have $P(s(h)xs(k))=P(x)$.
\end{lemma}

\begin{proof}
If $x \in s(h)^\perp\L^p(\cal{M})$ then we will show that $P(x)=0$. By Lemma \ref{Lemma-Han}, the element
$
\begin{bmatrix} 
|x^*|& x \\ 
x^* & |x|
\end{bmatrix} 
$
of $S^p_{2}(\L^p(\cal{M}))$ is positive. Note that $|x^*|$ belongs to $s(h)^\perp\L^p(\cal{M})$. 
Hence, we have
$$
\begin{bmatrix}
0   & P(x)\\
P^\circ(x^*) & P_2(|x|)
\end{bmatrix}
\ov{\eqref{Invariance-2}}{=} \begin{bmatrix}
P_1(|x^*|) & P(x)\\
P^\circ(x^*) & P_2(|x|)
\end{bmatrix}
\ov{\eqref{Phi-P1}}{=} \Phi\left(\begin{bmatrix}
|x^*|   & x \\
x^* & |x|
\end{bmatrix}\right).
$$
Since $\Phi$ is positive, this element is positive. Using Lemma \ref{Lemma-Matricial-inequality}, we infer that $P(x)=0$. If $x \in \L^p(\cal{M})s(k)^\perp$ then similarly we have $P(x)=0$. Now, if $x \in \L^p(\cal{M})$ we have
\begin{align*}
\MoveEqLeft
P(x)
=P((s(h)+s(h)^\perp)x(s(k)+s(k)^\perp)) \\        
&=P(s(h)xs(k))+P(s(h) x s(k)^\perp)+P(s(h)^\perp x s(k))+ P(s(h)^\perp x s(k)^\perp) \\
&=P(s(h)xs(k)).
\end{align*}
\end{proof}

\begin{remark} \normalfont
For any $x \in S^p_2(\L^p(\cal{M}))$, we deduce from Lemma \ref{final} combined with \eqref{Invariance-2} that 
$$
\Phi(s(H) xs(H))
=\Phi(x).
$$
\end{remark}

By applying Theorem \ref{Lifting-decomposable} to the $n$-pseudo-decomposable map $P \co \L^p(\cal{M}) \to \L^p(\cal{M})$, we see that there exists a linear map $w \co \cal{M} \to s(P_1(h))\cal{M}s(P_2(k))=s(h)\cal{M}s(k)$ such that
\begin{equation}
\label{relevement-E_h-dec}
\quad P\big(h^{\frac{1}{2}}xk^{\frac{1}{2}}\big)
\ov{\eqref{lifting-decomposable}}{=} h^{\frac{1}{2}} w(x) k^{\frac{1}{2}}, \qquad x \in \cal{M}.
\end{equation}
Moreover, this map $w$ is normal and contractively $n$-pseudo-decomposable. Let $u_1 \co \cal{M} \to s(h)\cal{M}s(h)$ and $u_2 \co \cal{M} \to s(k)\cal{M}s(k)$ be the liftings of $n$-positive maps $P_1$ and $P_2$ provided by Theorem \ref{Th-relevement-cp} with respect to the elements $h$ and $k$. We introduce the linear map $Q \ov{\mathrm{def}}{=}\begin{bmatrix}
u_1 & w \\
w^\circ  & u_2 \\
\end{bmatrix} \co \M_2(\cal{M}) \to \cal{N}$. For any element $\begin{bmatrix}
 x_{11} &  x_{12}  \\
 x_{21} &  x_{22}    \\
\end{bmatrix}$ of $\M_2(\cal{M})$, we have 
\begin{align}
\MoveEqLeft
\label{relevement-Phi}
\Phi\left(H^{\frac{1}{2}}\begin{bmatrix}
 x_{11} &  x_{12}  \\
 x_{21} &  x_{22}    \\
\end{bmatrix}H^{\frac{1}{2}}\right)
=\Phi\left(\begin{bmatrix}
h^{\frac{1}{2}} x_{11} h^{\frac{1}{2}} & h^{\frac{1}{2}} x_{12}k^{\frac{1}{2}}  \\
k^{\frac{1}{2}}  x_{21} h^{\frac{1}{2}}  & k^{\frac{1}{2}} x_{22} k^{\frac{1}{2}}   \\
\end{bmatrix}\right) \\
&\ov{\eqref{Phi-P1}}{=} \begin{bmatrix}
P_1(h^{\frac{1}{2}} x_{11} h^{\frac{1}{2}}) & P(h^{\frac{1}{2}} x_{12}k^{\frac{1}{2}})  \\
P^\circ(k^{\frac{1}{2}}  x_{21} h^{\frac{1}{2}})  & P_2(k^{\frac{1}{2}} x_{22} k^{\frac{1}{2}})   \\
\end{bmatrix} 
\ov{\eqref{equa-relevement}\eqref{relevement-E_h-dec}}{=} \begin{bmatrix}
h^{\frac{1}{2}}u_1(x_{11})h^{\frac{1}{2}}& h^{\frac{1}{2}}w(x_{12})k^{\frac{1}{2}} \\
k^{\frac{1}{2}}w^\circ(x_{21})h^{\frac{1}{2}}  & k^{\frac{1}{2}}u_2(x_{22})k^{\frac{1}{2}} \\
\end{bmatrix} \nonumber \\
&=H^{\frac{1}{2}}\begin{bmatrix}
u_1(x_{11})& w(x_{12}) \\
w^\circ(x_{21})  & u_2(x_{22}) \\
\end{bmatrix}H
=H^{\frac{1}{2}}Q\left(\begin{bmatrix}
x_{11}& x_{12} \\
x_{21}  & x_{22} \\
\end{bmatrix}\right)H^{\frac{1}{2}}. \nonumber
\end{align}
Consequently, the map $Q$ is the lifting of the $n$-positive map $\Phi$ provided by Theorem \ref{Th-relevement-cp} with respect to the positive element $H$. Hence this map $Q$ is unital, contractive, normal and $n$-positive.

\begin{lemma}
\label{Lemma-bis-2}
The restriction $Q|\cal{N} \co \cal{N} \to \cal{N}$ is a projection. 
\end{lemma}

\begin{proof}
For any $x \in \cal{N}$, we have 
$$
H^{\frac{1}{2}}Q(x)H^{\frac{1}{2}}
\ov{\eqref{relevement-Phi}}{=} \Phi\big(H^{\frac{1}{2}}xH^{\frac{1}{2}}\big)
=\Phi^2\big(H^{\frac{1}{2}}xH^{\frac{1}{2}}\big)
\ov{\eqref{relevement-Phi}}{=} \Phi\big(H^{\frac{1}{2}} Q(x) H^{\frac{1}{2}}\big)
\ov{\eqref{relevement-Phi}}{=} H^{\frac{1}{2}} Q^2(x) H^{\frac{1}{2}}.
$$
Using Lemma \ref{lemma2-GL}, we obtain the conclusion. 
\end{proof} 

\begin{lemma}
\label{Lemma-bis-2}
The projection $Q|\cal{N} \co \cal{N} \to \cal{N}$ is faithful.
\end{lemma}

\begin{proof}
If $x=\begin{bmatrix}
 x_{11} &  x_{12}  \\
 x_{21} &  x_{22}    \\
\end{bmatrix}$ is an element of $\cal{N}_+$ satisfying $Q(x)=0$, we have
$$
\Phi\big(H^{\frac{1}{2}}xH^{\frac{1}{2}}\big)
\ov{\eqref{relevement-Phi}}{=} H^{\frac{1}{2}} Q(x) H^{\frac{1}{2}}
=0.
$$
So by \eqref{relevement-Phi}, we have $P_1(h^{\frac{1}{2}} x_{11} h^{\frac{1}{2}})=0$ and $P_2(k^{\frac{1}{2}} x_{22} k^{\frac{1}{2}})=0$. By \cite[1.6.9]{Dix77}, we have $0 \leq x_{11} \leq \norm{x_{11}}_\infty$ so 
$$
0 \leq h^{\frac{1}{2}}x_{11}h^{\frac{1}{2}} \leq \norm{x_{11}}_{\infty}h.
$$ 
Recall that the norm of a noncommutative $\L^p$-space is strictly monotone. By Lemma \ref{Lemma-monotone}, we see that $h^{\frac{1}{2}}x_{11}h^{\frac{1}{2}}=0$. Since $x_{11}$ belongs to $s(h)\cal{M}s(h)$, we conclude by Lemma \ref{lemma2-GL} that $x_{11}=0$. 
Similarly, we show that $x_{22}=0$. By Lemma \ref{Lemma-Matricial-inequality} (see also \cite[Proposition 1.3.2]{Bha07}), we conclude that $x=0$.
\end{proof} 

Now, we consider the normal faithful positive linear form $\psi$ on the von Neumann algebra $\cal{N}$ defined by
$$
\psi(x) 
\ov{\mathrm{def}}{=} (\tr \ot \tr_{\varphi})(H^p x),\quad x \in \cal{N}.
$$

\begin{lemma} 
\label{Lemma-prservation-state}
The restriction $Q|\cal{N}$ preserves $\psi$ i.e. we have $\psi \circ Q=\psi$.
\end{lemma}  

\begin{proof}
Since $h$ is positive, by Lemma \ref{Lemma-J-lp}, we have $J_{\L^p(\cal{M})}(h)=\norm{h}^{2-p}_{p}h^{p-1}$. By \cite[Corollary 5.2]{PiX03}, the Banach space $\L^p(\cal{M}_h)$ is smooth and strictly convex. Using the contractive dual map $P_1^* \co \L^{p^*}(\cal{M}) \to \L^{p^*}(\cal{M})$ and Lemma \ref{Lemma-smooth}, we see that $P_1^*(\norm{h}^{2-p}_{p}h^{p-1})=\norm{h}^{2-p}_{p}h^{p-1}$, that is
\begin{equation}
\label{equa-inter-400}
P_1^*(h^{p-1})
=h^{p-1}.
\end{equation}
Similarly, we have $P_2^*(k^{p-1})
=k^{p-1}$. Using \cite[Lemma 3.3]{ArK23} in the second equality, we conclude that
\begin{align*}
\MoveEqLeft
\label{equa-inter-400}
\Phi^*(H^{p-1})
\ov{\eqref{Phi-P1}}{=} \begin{bmatrix}
   P_1  &  P \\
   P^\circ  &  P_2  \\
\end{bmatrix}^*\bigg(\begin{bmatrix}
   h  &  0 \\
   0  &  k \\
\end{bmatrix}^{p-1}\bigg)
=\begin{bmatrix}
   P_1^*  &  (P^\circ)^* \\
   P^*  &  P_2^*  \\
\end{bmatrix}\left(\begin{bmatrix}
   h^{p-1}  &  0 \\
   0  &  k^{p-1} \\
\end{bmatrix}\right) \\
&=\begin{bmatrix}
    P_1^*(h^{p-1})  &  0 \\
   0  &  P_2^*(k^{p-1}) \\
\end{bmatrix}
\ov{\eqref{equa-inter-400}}{=} \begin{bmatrix}
   h^{p-1}  &  0 \\
   0  &  k^{p-1} \\
\end{bmatrix}
=H^{p-1}.         
\end{align*}
For any $z \in \L^p(\cal{N})$, it follows that
\begin{equation}
\label{equa-inter-401}
(\tr \ot \tr_{\varphi})\big(H^{p-1}\Phi(z)\big)
=(\tr \ot \tr_{\varphi})\big(\Phi^*(H^{p-1})z\big)
=(\tr \ot \tr_{\varphi})(H^{p-1} z).
\end{equation}
In particular, for any $x \in \cal{N}$, we have
\begin{align*}
\MoveEqLeft
(\tr \ot \tr_{\varphi})\big(H^{p-1}(H^{\frac{1}{2}}Q(x) H^{\frac{1}{2}})\big)
\ov{\eqref{relevement-Phi}}{=} (\tr \ot \tr_{\varphi})\big(H^{p-1}\Phi(H^{\frac{1}{2}} x H^{\frac{1}{2}})\big) \\
&\ov{\eqref{equa-inter-401}}{=} (\tr \ot \tr_{\varphi})\big(H^{p-1}(H^{\frac{1}{2}} x H^{\frac{1}{2}})\big)
\end{align*}
that is $(\tr \ot \tr_{\varphi})\big(H^pQ(x)\big)=(\tr \ot \tr_{\varphi})(H^p x)$. Hence $\psi\big(Q(x)\big)=\psi(x)$.
\end{proof}

Note that $s(h)\cal{M}s(k)$ is a $\W^*$-TRO (hence a $\JBW^*$-triple). So we can use Theorem \ref{Th-youngson} and Theorem \ref{Th-proj-JBW} for the description of the range of the restriction of the map $w$. Consequently we have proved the ``only if'' part of Theorem \ref{main-th-ds-bis}.

\vspace{0.2cm}

Conversely, suppose that the conditions of Theorem \ref{main-th-ds-bis} are satisfied. We introduce the normal faithful positive linear form $\psi$ on the von Neumann algebra $\cal{N} \ov{\mathrm{def}}{=} s(H)\M_2(\cal{M})s(H)$ defined as  the restriction of $(\tr \ot \tr_{\varphi})(H^p\,\cdot)$. Recall that $D_{\psi} \in \L^1(\cal{N},\psi)$ denotes the density operator associated with $\psi$. From \eqref{kappa}, we have a canonical map  $\kappa \co \L^p(\cal{N},\psi) \to \L^p(\cal{N},(\tr \ot \varphi)_{s(H)})$ which induces a completely order and isometric identification. If $x \in \cal{N}$, using Lemma \ref{Lemme-trace coincide} in the third equality, we see that
\begin{align*}
\MoveEqLeft
\label{Magic-Formula}
\tr_{\psi}(D_{\psi} x) 
\ov{\eqref{varphi-trace}}{=} \psi(x) 
= (\tr \ot \tr_{\varphi})(H^p x) 
=(\tr \ot \tr_{\varphi_{s(h)}})(H^p x) \\
&\ov{\eqref{Trace-preserving}}{=} \tr_{\psi}(\kappa^{-1}(H^p x))
=\tr_{\psi}(\kappa^{-1}(H)^p x).         
\end{align*}
We conclude that
\begin{equation}
\label{Magic-formula-2}
D_{\psi}
=\kappa^{-1}(H)^p.
\end{equation}
With the fourth condition of Theorem \ref{main-th-ds-bis}, we can consider by \eqref{Map-extension-Lp} with $\cal{N}$ instead of $\cal{M}$ the contractive $n$-positive operator $Q_p \co \L^p(\cal{N},\psi) \to \L^p(\cal{N},\psi)$ induced by the restriction $Q|\cal{N} \co \cal{N} \to \cal{N}$ of $Q$ and defined by
\begin{equation}
\label{Def-esperance-h}
Q_p\big(D_{\psi}^{\frac{1}{2p}} xD_{\psi}^{\frac{1}{2p}}\big)
\ov{\eqref{Map-extension-Lp}}{=} D_{\psi}^{\frac{1}{2p}}Q(x)D_{\psi}^{\frac{1}{2p}},\quad x \in \cal{N}. 
\end{equation}
For any $x \in \cal{N}$, note that
\begin{align*}
\MoveEqLeft
Q_p^2\big(D_{\psi}^{\frac{1}{2p}} x D_{\psi}^{\frac{1}{2p}}\big)
\ov{\eqref{Def-esperance-h}}{=} Q_p\big(D_{\psi}^{\frac{1}{2p}} Q(x) D_{\psi}^{\frac{1}{2p}}\big)
\ov{\eqref{Def-esperance-h}}{=} D_{\psi}^{\frac{1}{2p}} Q^2(x) D_{\psi}^{\frac{1}{2p}} \\
&=D_{\psi}^{\frac{1}{2p}} Q(x)D_{\psi}^{\frac{1}{2p}}
\ov{\eqref{Def-esperance-h}}{=} Q_p\big(D_{\psi}^{\frac{1}{2p}} x D_{\psi}^{\frac{1}{2p}}\big).
\end{align*}
We deduce that $Q_p^2=Q_p$, i.e. that the map $Q_p$ is a projection. 

We also consider the restrictions $\psi_1$ and $\psi_2$ of the normal positive linear forms $\tr_{\varphi}(h^p\, \cdot)$ and $\tr_{\varphi}(k^p\, \cdot)$ on the von Neumann algebras $s(h)\cal{M}s(h)$ and $s(k)\cal{M}s(k)$. With the canonical identifications $\kappa_1 \co \L^p(s(h)\cal{M}s(h),\psi_1) \to \L^p(s(h)\cal{M}s(h),\varphi_{s(h)})$ and $\kappa_2 \co \L^p(s(k)\cal{M}s(k),\psi_2) \to \L^p(s(k)\cal{M}s(k),\varphi_{s(k)})$, we have similarly to \eqref{Magic-formula-2} the equalities
\begin{equation}
\label{rel-fin99}
D_{\psi_1}
=\kappa_1^{-1}(h)^p
\quad \text{and} \quad
D_{\psi_2}
=\kappa_2^{-1}(k)^p.
\end{equation}
Using again the fourth condition of Theorem \ref{main-th-ds-bis}, we see that for any $x \in \cal{N}$
$$ 
\psi_1(u_1(x))
=\tr_{\varphi}(h^px)=(\tr \ot \tr_{\varphi})\left(H^p Q\left(\begin{bmatrix}
 x &  0  \\
 0 &  0    \\
\end{bmatrix}\right)\right)
=(\tr \ot \tr_{\varphi})\left(H^p \begin{bmatrix}
 x &  0  \\
 0 &  0    \\
\end{bmatrix}\right)
=\psi_1(x).
$$
We deduce that $u_1|s(h)\cal{M}s(h) \co s(h)\cal{M}s(h) \to s(h)\cal{M}s(h)$ preserves $\psi_1$. Similarly, the linear map $u_2|s(k)\cal{M}s(k) \co s(k)\cal{M}s(k) \to s(k)\cal{M}s(k)$ preserves $\psi_2$. By \eqref{Map-extension-Lp}, these maps admit contractive extensions $\tilde{P}_1 \co \L^p(s(h)\cal{M}s(h),\psi_1) \to \L^p(s(h)\cal{M}s(h),\psi_1)$ and $\tilde{P}_2 \co \L^p(s(k)\cal{M}s(k),\psi_2) \to \L^p(s(k)\cal{M}s(k),\psi_2)$ such that
\begin{equation}
\label{ref-final-5}
\tilde{P}_1\big(D_{\psi_1}^{\frac{1}{2p}} xD_{\psi_1}^{\frac{1}{2p}}\big)
\ov{\eqref{Map-extension-Lp}}{=} D_{\psi_1}^{\frac{1}{2p}}u_1(x)D_{\psi_1}^{\frac{1}{2p}}
\quad \text{and} \quad
\tilde{P}_2\big(D_{\psi_2}^{\frac{1}{2p}} xk_{\psi_2}^{\frac{1}{2p}}\big)
\ov{\eqref{Map-extension-Lp}}{=} D_{\psi_1}^{\frac{1}{2p}}u_2(x)D_{\psi_1}^{\frac{1}{2p}}.
\end{equation}
We let $P_1 \ov{\mathrm{def}}{=} \kappa_1 \tilde{P}_1\kappa_1^{-1}$ and $P_2 \ov{\mathrm{def}}{=} \kappa_2 \tilde{P}_2\kappa_2^{-1}$. We have
\begin{align}
\MoveEqLeft
\label{ref-final-4}
P_1\big(h^{\frac{1}{2}}x h^{\frac{1}{2}}\big)
=\kappa_1 \tilde{P}_1\kappa_1^{-1}\big(h^{\frac{1}{2}}x h^{\frac{1}{2}}\big) 
\ov{\eqref{rel-fin99}}{=}\kappa_1 \tilde{P}_1 \big(D_{\psi_1}^{\frac{1}{2p}}x D_{\psi_1}^{\frac{1}{2p}}\big)  \\      
&\ov{\eqref{ref-final-5}}{=} \kappa_1 (D_{\psi_1}^{\frac{1}{2p}}u_1(x)D_{\psi_1}^{\frac{1}{2p}})
\ov{\eqref{rel-fin99}}{=} h^{\frac{1}{2}}u_1(x) h^{\frac{1}{2}} \nonumber
\end{align}
and a similar equality for $P_2$. For any $x \in \cal{N}$, we obtain
\begin{align*}
\MoveEqLeft
\kappa Q_p\kappa^{-1}\big(H^{\frac{1}{2}}x H^{\frac{1}{2}}\big)         
=\kappa Q_p\big(\kappa^{-1}(H)^{\frac{1}{2}}x \kappa^{-1}(H)^{\frac{1}{2}}\big)   
\ov{\eqref{Magic-formula-2}}{=} \kappa Q_p\big(D_{\psi}^{\frac{1}{2p}} x D_{\psi}^{\frac{1}{2p}}\big) \\
&\ov{\eqref{Def-esperance-h}}{=} \kappa (D_{\psi}^{\frac{1}{2p}}Q(x)D_{\psi}^{\frac{1}{2p}}) 
\ov{\eqref{Magic-formula-2}}{=} H^{\frac{1}{2}}Q(x) H^{\frac{1}{2}}
=\begin{bmatrix}
h^{\frac{1}{2}}u_1(x_{11})h^{\frac{1}{2}}& h^{\frac{1}{2}}w(x_{12})k^{\frac{1}{2}} \\
k^{\frac{1}{2}}w^\circ(x_{21})h^{\frac{1}{2}}  & k^{\frac{1}{2}}u_2(x_{22})k^{\frac{1}{2}} \\
\end{bmatrix} \\
&\ov{\eqref{relevement-dec-intro}\eqref{ref-final-4}}{=} \begin{bmatrix}
P_1(h^{\frac{1}{2}} x_{11} h^{\frac{1}{2}}) & P(h^{\frac{1}{2}}x_{12}k^{\frac{1}{2}}) \\
P^\circ(k^{\frac{1}{2}} x_{21} h^{\frac{1}{2}})  & P_2(k^{\frac{1}{2}}x_{22} k^{\frac{1}{2}}) \\
\end{bmatrix}
=\begin{bmatrix}
P_1 & P\\
P^\circ  & P_2 \\
\end{bmatrix}\big(H^{\frac{1}{2}}x H^{\frac{1}{2}}\big) .
\end{align*}
By density, with Lemma \ref{lemma2-GL}, we conclude that
$$
\kappa Q_p\kappa^{-1}
=\begin{bmatrix}
P_1 & P|_{s(h)\L^p(\cal{M})s(k)}\\
P|_{s(h)\L^p(\cal{M})s(k)}^\circ  & P_2 \\
\end{bmatrix}.
$$
Since $s(h)$ and $s(k)$ belong to the centralizer of $\varphi$, the projection $s(H)$ belongs to the centralizer of $\tr \ot \varphi$. So we have an order isometric identification of $\L^p\big(\cal{N},(\tr \ot \varphi)_{s(H)}\big)$ as a subspace of the Banach space $\L^p(\M_2(\cal{M}),\phi)$. By Lemma \ref{Mult-decomposable}, the linear map
\begin{equation*}
\label{Matrice-2-2-Phi-bis}
\Phi
\co S^p_2(\L^p(\cal{M})) \to \L^p(\cal{N},(\tr \ot \varphi)_{s(H)}), \quad \begin{bmatrix}
   x  &  y \\
   z &  t  \\
\end{bmatrix}\mapsto 
\begin{bmatrix}
   s(h)xs(h)  &  s(h)ys(k) \\
   s(k)zs(h)  &  s(k)ts(k)  \\
\end{bmatrix}.
\end{equation*}
is completely positive. Using the first point of Theorem \ref{main-th-ds-bis} in the second equality, we obtain that
$$
\kappa Q_p\kappa^{-1}\Phi
=\begin{bmatrix}
P_1 & P|s(h)\cal{M}s(k)\\
P|s(h)\cal{M}s(k)^\circ  & P_2 \\
\end{bmatrix}\Phi
=\begin{bmatrix}
P_1(s(h)\cdot s(h)) & P\\
P^\circ  & P_2(s(k)\cdot s(k)) \\
\end{bmatrix}
$$
By composition, this map is $n$-positive. Furthermore, we have $\norm{P_1(s(h)\,\cdot\, s(h))}_{\L^p(\cal{M}) \to \L^p(\cal{M})} \leq 1 $ and $\norm{P_2(s(k)\,\cdot\, s(k))}_{\L^p(\cal{M}) \to \L^p(\cal{M})} \leq 1$. We deduce by \eqref{Norm-dec} that the map $P$ is a contractively $n$-pseudo-decomposable projection. The proof is complete. 
\end{proof}

\begin{remark} \normalfont
The case where $\cal{M}$ is not $\sigma$-finite is beyond the scope of this paper. Probably, it suffices to adapt the painful method of \cite{ArR19}. 
\end{remark}

\begin{remark} \normalfont
The author believes that with slight modifications, we can prove the case $p=1$ as in \cite{ArR19}. 
\end{remark}

\begin{remark} \normalfont
If $n=\infty$, we do not know if $w(s(h)\cal{M}s(k))$ is a sub-TRO of the $\W^*$-TRO $s(h)\cal{M}s(k)$. In this case, $w$ would be a TRO-conditional expectation, see Remark \ref{Rem-TRO-cond}. More generally, if $n \geq 1$, we do not know if $w(s(h)\cal{M}s(k))$ is a sub-triple of $s(h)\cal{M}s(k)$. However, it is written in \cite[Lemma 4.1]{BuP02} \cite[Lemma 5.3]{EdR96} that there exists a $\JW^*$-subtriple $X$ of $s(h)\cal{M}s(k)$ such that $X$ is linearly isometric to $w(s(h)\cal{M}s(k))$ and such that $X$ is the range of a weak*-continuous projection on $s(h)\cal{M}s(k)$.
\end{remark}

\section{$\L^p$-spaces associated to $\sigma$-finite $\JBW^*$-triples}
\label{Sec-Lp-JBW}

In this section, we introduce $\L^p$-spaces associated to some suitable $\JBW^*$-triples and we connect these spaces to the nonassociative $\L^p$-spaces introduced in our previous paper \cite{Arh23}. It is important to note that a $\JBW^*$-triple does not admit an involution, contrarily to the case of $\JBW^*$-algebras.

We will use the following result \cite[Corollary 5.3]{Wer18} on complex interpolation. See  \cite{CoS98}, \cite{HaP89}, \cite{LiP64}, \cite{Wat95} and \cite{Wat00}   
for variants of this result. Here, we use the notation $\ovl{Y}$ for the complex conjugate of a Banach space $Y$, that is the same set, with the same addition and norm but with the conjugate multiplication by a scalar. The elements of the Banach space $\ovl{Y}$ are denoted $\ovl{y}$ where $y \in Y$. If $u \co Y \to Y$ is a linear map, we can define a linear map $\ovl{u} \co \ovl{Y} \to \ovl{Y}$ by $\ovl{u}(\ovl{y}) \ov{\mathrm{def}}{=} \ovl{u(y)}$.

\begin{thm} 
\label{Th-Werner-bis}
Let $Y$ be a Banach space. Let $v \co \cal{H} \to Y$ be a contractive injective map from a Hilbert space $\cal{H}$ into $Y$ with dense range. The composition $v \circ \ovl{v^*} \co \ovl{Y^*} \to Y$ defines an interpolation couple and we have $(\ovl{Y^*},Y)_{\frac{1}{2}}
=\cal{H}$. Moreover, for any $0 < \theta < 1$ we have $\ovl{(\ovl{Y^*},Y)_\theta^*}=(\ovl{Y^*},Y)_{1-\theta}$.
\end{thm}

\paragraph{Support projection of a linear functional} 
Consider a $\JBW^*$-algebra $\cal{M}$ equipped with a normal positive linear functional $\varphi$. The support idempotent $e$ of $\varphi$ is the unique projection $e$ of $\cal{M}$ such that $\{x \in  \cal{M} : \varphi(x^* \circ x) = 0\} = U_{1-e}(\cal{M})$, where 
\begin{equation}
\label{def-Ua}
U_x(a) 
\ov{\mathrm{def}}{=} \{x,a,x\}, \quad a \in \cal{M}
\end{equation}
see \cite[Definition 5.10.18 p.~284]{CGRP18}. We have $\varphi(e)=\norm{\varphi}$ by \cite[Proposition 5.10.20 p.~284]{CGRP18}.

\paragraph{Support tripotent of a linear functional} Assume that $X$ is a $\JBW^*$-triple and consider some linear functional $\varphi \in X_*$. By \cite[Definition 5.10.58 p.~300]{CGRP18} \cite[Proposition 2]{FrB85b} there exists a unique tripotent $s(\varphi) \in X$, called the support tripotent of $\varphi$, such that
\begin{enumerate}
\item $\varphi=\varphi\circ P_2(s(\varphi))$,
\item $\varphi|_{X_2(s(\varphi))}$ is a faithful weak* continuous positive functional on the $\JBW^*$-algebra $X_2(s(\varphi))$,
\end{enumerate}
where $P_2(s(\varphi)) \co X \to X$ is the Peirce projection with range $X_2(s(\varphi))$ defined in \eqref{Peirce-1}. Furthermore, we have
\begin{equation}
\label{norm-varphi}
\norm{\varphi}
=\varphi(s(\varphi))=\norm{\varphi|X_2(s(\varphi))}
\end{equation}
and $s(\varphi)$ is the support idempotent of $\varphi|X_2(s(\varphi))$ 
in the $\JBW^*$-algebra $X_2(s(\varphi))$. 

It is essentially shown in \cite[part $(b)$ in the proof of Proposition 2]{FrB85b} and \cite[p.~300]{CGRP18} that 
\begin{equation}
\label{supporttripotents}
\text{if $u$ is a tripotent in $X$ with $\norm{\varphi} = \varphi (u)$, then $s(\varphi) \leq u$.}
\end{equation}

\begin{example} \normalfont
\label{algebras-1}
Consider a $\JBW^*$-algebra $\cal{M}$ equipped with a normal positive linear functional $\varphi$. The support tripotent $s(\varphi)$ of $\varphi$ is equal to the support idempotent $e$ of $\varphi$. Moreover, the support tripotent $s(\varphi)=e$ is complete if and only if $e=1$. By \cite[Proposition 5.10.20 (ii) p. 284]{CGRP18}, this is equivalent to the faithfulness of the functional $\varphi$.

\vspace{0.2cm}

\begin{proof}
We can suppose that $\varphi \not=0$. By \cite[Proposition 5.10.20 (ii) p. 284]{CGRP18}, we have $\varphi(e)=\norm{\varphi}$. Consequently by \eqref{supporttripotents}, we deduce that $s(\varphi) \leq e$. By Lemma \ref{Lemma-smaller}, we see in particular that $s(\varphi)$ is a projection of the $\JBW^*$-algebra $X_2(e)$. We have $\varphi(s(\varphi)) \ov{\eqref{norm-varphi}}{=} \norm{\varphi}=\norm{\varphi} \norm{s(\varphi)}$. By \cite[Proposition 5.10.20 (iii) p. 284]{CGRP18}, we infer that 
$$
\{e,s(\varphi),e\}\ov{\eqref{def-Ua}}{=} U_e(s(\varphi))
=\norm{s(\varphi)}e
=e.
$$ 
Using again Lemma \ref{Lemma-smaller}, we deduce that $e \leq s(\varphi)$. We conclude that $e = s(\varphi)$.

If $e=1$ then $e$ is a unitary. So it is complete by \cite[(1.21) (ii)]{HoN88}. In the opposite direction, if the projection $s(\varphi)=e$ is complete then we have $
P_0(e)(1-e)
\ov{\eqref{Pierce-example-fo}}{=} \{1-e,1-e,1-e\}=1-e$ since $1-e$ is a projection, hence a tripotent. Consequently, $1-e$ belongs to the Peirce-0 subspace $\cal{M}_0(e)$ which is equal to $\{0\}$ since $e$ is complete. We conclude that $e=1$.
\end{proof}
\end{example}

\paragraph{Sesquilinear form associated to a functional}
In \cite[Proposition 1.2]{BaF87} \cite[Proposition 5.10.60 p.~301]{CGRP18}, the authors showed that given an element $\varphi \in X_*$ and an element $z \in X$ such that $\varphi(z)=\|\varphi\|=\norm{z}=1$, we have a positive hermitian sesquilinear form
\begin{equation}
\label{Def-scalar-varphi-bis}
\langle x, y \rangle_{\cal{H}_\varphi}
\ov{\mathrm{def}}{=} \varphi(\{x,y,z\}), \quad x,y \in X
\end{equation} whose associated seminorm 
\begin{equation}
\label{norm-svarphi}
\norm{x}_{\cal{H}_\varphi}
\ov{\mathrm{def}}{=} \sqrt{\varphi(\{x,x,z\})},\quad x \in X
\end{equation}
is independent of $z$. Moreover, by \cite[Lemma 4.1]{EdR98a} the kernel of the seminorm $\norm{\cdot}_{\cal{H}_\varphi}$ is precisely $X_0(s(\varphi))$, that is
\begin{equation}
\label{kernel-seminorm} 
\{ x \in X : \norm{x}_{\cal{H}_\varphi}=0\}
= X_0(s(\varphi)).
\end{equation}

\paragraph{$\sigma$-finite $\JBW^*$-triples}
Following \cite[p.~293]{EdR98a}, we say that a tripotent $u$ in a $\JBW^*$-triple is $\sigma$-finite if any family of pairwise orthogonal nonzero smaller tripotents is at most countable. Clearly, 0 is a $\sigma$-finite tripotent. A $\JBW^*$-triple $X$ is said to be $\sigma$-finite if every tripotent $u$ in $X$ is $\sigma$-finite. By \cite[Theorem 3.2]{EdR98a}, this is is equivalent to saying that every orthogonal subset of tripotents in $X$ is at most countable.

We have the following characterization \cite[Theorem 3.2]{EdR98a} of $\sigma$-finiteness.  

\begin{lemma}
\label{support-sigma-fini}
A tripotent $u$ of a $\JBW^*$-triple $X$ is $\sigma$-finite if and only if it is equal to the support tripotent $s(\varphi)$ for some linear functional $\varphi \in X_*$.
\end{lemma}

\begin{example} \normalfont
\label{Ex-algebra-proj}
Recall that an element $e$ of a $\JBW^*$-algebra is said to be a projection if $e^* = e$ and $e \circ e = e$ and that projections $e$ and $f$ are called orthogonal if $e \circ f = 0$. By \cite[Lemma 7.2 (b)]{HKPP20}, a projection $e$ is a $\sigma$-finite (i.e. any family of pairwise orthogonal smaller nonzero projections is at most countable) if and only if it is a $\sigma$-finite tripotent. 
\end{example}


We will use the following result \cite[Lemma 4.2]{EdR98a}. Recall that $X_2(u)$ and $X_0(v)$ are the ranges of the Peirce projections $P_2(u)$ and $P_0(v)$ defined in \eqref{Peirce-1} and \eqref{Peirce-2}.

\begin{lemma}
\label{Lemma-Edwards}
Let $X$ be a $\JBW^*$-triple and let $v$ be a $\sigma$-finite tripotent in $X$. Let $u$ be a tripotent in $X$ such that $X_2(u)_{\sa} \cap X_0(v)=\{0\}$. Then $u$ is $\sigma$-finite.
\end{lemma}

Now, we give the following characterization of $\sigma$-finite $\JBW^*$-triples. We refer to \cite[Theorem 4.4]{EdR98a} for a lot of other characterizations of $\sigma$-finite $\JBW^*$-triples.

\begin{prop}
\label{Prop-carac-sigma}
A $\JBW^*$-triple $X$ is $\sigma$-finite if and only if $X$ admits a linear functional $\varphi \in X_*$ such that the support tripotent $s(\varphi)$ is complete.
\end{prop}

\begin{proof}
$\Rightarrow$: By \cite[Fact 5.7.25 p.~226]{CGRP18}, there exists a complete tripotent $u$, which is $\sigma$-finite since $X$ is $\sigma$-finite. With Lemma \ref{support-sigma-fini}, we deduce that $u$ is equal to the support tripotent $s(\varphi)$ for some linear functional $\varphi \in X_*$.

$\Leftarrow$: Suppose that $X$ admits a functional $\varphi \in X_*$ such that the support tripotent $s(\varphi)$ is complete, i.e. $X_0(s(\varphi))=\{0\}$. Lemma \ref{support-sigma-fini} says that $s(\varphi)$ is $\sigma$-finite. Now, if $u$ is a tripotent in $X$ then $X_2(u)_{\sa} \cap X_0(s(\varphi))=\{0\}$. We deduce with Lemma \ref{Lemma-Edwards} that $u$ is $\sigma$-finite. We conclude that $X$ is $\sigma$-finite.
\end{proof}

\begin{example} \normalfont
Recall that a $\JBW^*$-algebra $\cal{M}$ is $\sigma$-finite if any orthogonal family of non-zero projections is countable. By \cite{Arh23}, a $\JBW^*$-algebra is $\sigma$-finite if and only if there exists a normal faithful state $\varphi$. A $\JBW^*$-algebra $\cal{M}$ is $\sigma$-finite as a $\JBW^*$-algebra if and only if it is $\sigma$-finite as a $\JBW^*$-triple. 
\end{example}

\begin{proof}
$\Rightarrow$: If $\cal{M}$ is $\sigma$-finite as a $\JBW^*$-algebra then there exists a normal faithful state. By Example \ref{algebras-1}, its support tripotent is complete. Consequently, $\cal{M}$ admits a linear functional $\varphi \in X_*$ such that the support tripotent $s(\varphi)$ is complete. We conclude with Proposition \ref{Prop-carac-sigma}.

$\Leftarrow$: Suppose that $\cal{M}$ is $\sigma$-finite as a $\JBW^*$-triple. Every orthogonal subset of tripotents in $A$ is at most countable. Hence every orthogonal subset of projections of $\cal{M}$ is at most countable by Example \ref{order-algebras}.
\end{proof}

\paragraph{$\L^p$-spaces of a $\sigma$-finite $\JBW^*$-triple} Let $X$ be a $\JBW^*$-triple equipped with some functional $\varphi \in X_*$ such that $\norm{\varphi}_{X_*}=1$. We will introduce some embedding of $X$ in its predual $X_*$. For that, we define for any $y \in X$ the linear functional $\varphi_y \co X \to \mathbb{C}$ by
\begin{equation}
\label{def-varphix}
\varphi_y(x)
\ov{\mathrm{def}}{=} \varphi(\{x,y,s(\varphi)\}), \quad x \in X.
\end{equation}

\begin{prop}
\label{prop-contract}
Suppose that the support tripotent $s(\varphi)$ is complete. For any $y \in X$, the conjugate linear map $i \co X \to X_*$, $y \mapsto \varphi_y$ is well-defined, injective and contractive. Moreover, its range is dense in the Banach space $X_*$.
\end{prop}

\begin{proof}
Recall that the triple product of a $\JBW^*$-triple is separately weak* continuous, see e.g. \cite[Theorem 5.7.20 p.~224]{CGRP18}. Hence for any $y \in X$ the map $x \mapsto \{x,y,s(\varphi)\}$ is weak* continuous. So by composition the functional $\varphi_y$ is weak* continuous. In a $\JB^*$-triple, the triple product is contractive, see \cite[Corollary 7.1.7 p.~440]{CGRP18}, \cite[Corollary 3]{FrB86} and \cite[p.~215]{Chu12}. Consequently for any $x,y \in X$, we have
$$
|\varphi_y(x)|
\ov{\eqref{def-varphix}}{=} |\varphi(\{x,y,s(\varphi)\})| 
\leq \norm{\{x,y,s(\varphi)\}}_X
\leq \norm{x}_X \norm{y}_X \norm{s(\varphi)}_X
\leq \norm{x}_X \norm{y}_X.
$$
We deduce that the map $i$ is contractive. Let $y \in X$. Suppose $\varphi_y=0$. We have
$$
\norm{y}_{\cal{H}_\varphi}^2
\ov{\eqref{norm-svarphi}}{=} \varphi(\{y,y,s(\varphi)\})
\ov{\eqref{def-varphix}}{=} \varphi_y(y)
=0.
$$
Since the tripotent $s(\varphi)$ is complete, we have $X_0(s(\varphi))=\{0\}$. By \eqref{kernel-seminorm}, we see that $\norm{\cdot}_{\cal{H}_\varphi}$ is a norm on $X$. We obtain that $y=0$. We conclude that the linear map $i$ is injective.

If $x \in X$ belongs the annihilator $i(X)^\perp$ of $i(X)$, we have 
$$
\norm{x}_{\cal{H}_\varphi}^2
\ov{\eqref{norm-svarphi}}{=} \varphi(\{x,x,s(\varphi)\}))
\ov{\eqref{def-varphix}}{=} \varphi_{x}(x)
=\la i(x), x \ra_{X_*,X}
=0.
$$ 
Using again the fact that $\norm{\cdot }_\varphi$ is a norm, we infer that $x=0$. We obtain that $i(X)^\perp=\{0\}$. By \cite[Proposition 1.10.15 (c) p. 93]{Meg98}, we conclude that the range $i(X)$ of $i$ is dense in the space $X_*$. 
\end{proof}

If $X$ is a complex Banach space then the map $j \co \ovl{X} \to X_*$, $\ovl{y} \mapsto \varphi_y$ is linear. Hence $(\ovl{X},X_*)$ is an interpolation couple of complex Banach spaces. In the spirit of the construction of noncommutative $\L^p$-spaces of Kosaki \cite{Kos84} and nonassociative $\L^p$-spaces of \cite{Arh23}, we introduce the following definition by using the complex interpolation method.

\begin{defi}
\label{defi-Lp}
Let $X$ be a $\JBW^*$-triple equipped with a functional $\varphi \in X_*$ with $\norm{\varphi}_{X_*}=1$ such that the support tripotent $s(\varphi)$ is complete. For any $1 < p < \infty$, we let
$$
\L^p(X,\varphi)
\ov{\mathrm{def}}{=} (\ovl{X},X_*)_{\frac{1}{p}}.
$$
We say that this space is the $\L^p$-space associated with $X$ and $\varphi$.
\end{defi}
We shall write simply $\L^p(X)$ when there is no ambiguity on the functional $\varphi$. Note that the sum $\ovl{X}+X_*$ is isometric to $X_*$ and that the intersection  $\ovl{X} \cap X_*=j(\ovl{X})$, whose norm is defined by \eqref{norm-intersection}, is isometric to $\ovl{X}$. By \cite[Proposition 2.4 p.~50]{Lun18}, we have the contractive inclusions $\ovl{X} \subset \L^p(X) \subset X_*$. Moreover, by \cite[Theorem 4.2.2 p. 91]{BeL76} the subspace $\ovl{X}$ is dense in the Banach space $\L^p(X)$. We let $\L^1(X,\varphi) \ov{\mathrm{def}}{=} X_*$.

\begin{example} \normalfont
Let $\cal{M}$ be a $\JBW^*$-algebra with product $\circ$ equipped with a normal faithful state $\varphi$. Recall that we can see $\cal{M}$ as a $\JBW^*$-triple, see  Example \ref{Ex-JBW-algebra-triple}. In \cite{Arh23}, we introduced the embedding $i \co \cal{M} \to \cal{M}_*$, $y \mapsto \varphi(y \circ \cdot)$ for defining nonassociative $\L^p$-spaces associated to $(\cal{M},\varphi)$. We will describe the link between these spaces and the ones of Definition \ref{defi-Lp}.

By Example \ref{algebras-1}, the support idempotent of the functional $\varphi$ is equal to the unit 1 of the $\JBW^*$-algebra $\cal{M}$ and equal to its support tripotent $s(\varphi)$. Now, for any $x,y \in \cal{M}$, we have
\begin{equation}
\label{Inter-897}
\varphi_y(x)
\ov{\eqref{def-varphix}}{=} \varphi(\{x,y,s(\varphi)\})
=\varphi(\{x,y,1\})
\ov{\eqref{triple-prod-JBWstar}}{=} 
\varphi(y^* \circ x).
\end{equation}
Moreover, the \textit{linear} map $\Phi \co \ovl{\cal{M}} \to \cal{M}$, $\ovl{y} \mapsto y^*$ is an isometric isomorphism by \eqref{def-JBstar}. Furthermore, for any $y \in \cal{M}$, observe that 
$$
i(\Phi(\ovl{y}))
=i(y^*)
=\varphi(y^* \circ \cdot)
\ov{\eqref{Inter-897}}{=} \varphi_y
=j(\ovl{y}).
$$
We deduce the following commutative diagram.
$$
\xymatrix @R=1cm @C=2cm{
    \cal{M} \ar[r]^{i}   & \cal{M}_*    \\
   \ovl{\cal{M}} \ar[ru]_{j} \ar[u]^{\Phi} &  
}
$$
By \cite[Remark 2.7.8 p.~59]{Pis03}, we conclude that there exist isometric isomorphisms between the spaces of Definition \ref{defi-Lp} and the ones of \cite{Arh23}. In other words, the spaces defined in Definition \ref{defi-Lp} generalize the nonassociative $\L^p$-spaces defined in \cite{Arh23}.
\end{example}

Let $X$ be a $\JBW^*$-triple equipped with a functional $\varphi \in X_*$ such that $s(\varphi)$ is complete. The formula \eqref{Def-scalar-varphi-bis} defines a complex scalar product. We denote by $\cal{H}_\varphi$ the associated Hilbert space, that is the completion of $X$ for the associated norm \eqref{norm-svarphi}. In the next result, we will identify the space $\L^2(X)$ to this Hilbert space and will describe the dual of the Banach space $\L^p(X)$.

\begin{prop}
\label{prop-L2-nonassociative}
Let $X$ be a $\JBW^*$-triple equipped with a functional $\varphi \in X_*$ with $\norm{\varphi}_{X_*}=1$ such that its support tripotent $s(\varphi)$ is complete. 
\begin{enumerate}
	\item The Banach space $\L^2(X)$ is linearly isometric to the Hilbert space $\cal{H}_\varphi$. More precisely, the norms of $\L^2(X)$ and of $\cal{H}_\varphi$ coincide on their common linear subspace $X$.
	\item Suppose $1 < p < \infty$. We have $\ovl{(\L^p(X))^*}=\L^{p^*}(X)$ isometrically.
\end{enumerate}
\end{prop}

\begin{proof}
Since the triple product is contractive \cite[Corollary 7.1.7 p.~440]{CGRP18} \cite[Corollary 3]{FrB86}, we have for any $x \in X$
$$
\norm{x}_{\cal{H}_\varphi}
\ov{\eqref{norm-svarphi}}{=} \sqrt{\varphi(\{x,x,s(\varphi)\})}
\leq \sqrt{\norm{\varphi} \norm{\{x,x,s(\varphi)\})}}
\leq \sqrt{\norm{x}_X \norm{x}_X \norm{s(\varphi)}_X}
=\norm{x}_X.
$$
So, the canonical map $u \co X \to \cal{H}_\varphi$ is contractive and injective with dense range, hence with weak* dense range. It is easy to check that $u$ is weak* continuous. Indeed, consider a \textit{bounded} net $(x_i)$ of $X$ such that $x_i \to x$ for the weak* topology of the the dual Banach space $X$. Since the triple product of a $\JBW^*$-triple is separately weak* continuous \cite[Theorem 5.7.20 p. 224]{CGRP18}, we have for any $y \in X$
$$
\la x_i, y \ra_{\cal{H}_\varphi}
\ov{\eqref{Def-scalar-varphi-bis}}{=} \varphi(\{x_i,y,s(\varphi)\})
\xra[i]{} \varphi(\{x,y,s(\varphi)\})
\ov{\eqref{Def-scalar-varphi-bis}}{=} \la x, y\ra_{\cal{H}_\varphi}.
$$ 
Since $X$ is dense in the space $\cal{H}_\varphi$ and since the net is bounded, it converges to $x$ for the weak* topology of $\cal{H}_\varphi$ by \cite[Exercise~2.71 p.~234]{Meg98}. With \cite[Corollary 5.1.20 p.~7]{CGRP18}, we conclude that $u$ is weak* continuous. By \cite[Theorem 3.1.17 p.~290]{Meg98}, its preadjoint $v \ov{\mathrm{def}}{=} u_* \co (\cal{H}_{\varphi})_* \to X_*$ is injective and has dense range. Note that $(\cal{H}_{\varphi})_*=\ovl{\cal{H}}_{\varphi}$. For any $x,y \in X$, we have
\begin{equation}
\label{Inter-102}
\varphi(\{x,y,s(\varphi)\})
=\la u(x),\ovl{y} \ra_{\cal{H}_\varphi,\ovl{\cal{H}}_\varphi}
=\la x,u_*(\ovl{y}) \ra_{X,X_*}.
\end{equation}
Hence $
v(\ovl{y})
=u_*(\ovl{y})
\ov{\eqref{Inter-102}}{=} \varphi(\{\cdot,y,s(\varphi)\})
\ov{\eqref{def-varphix}}{=} \varphi_y$. So
$$
v \circ \ovl{v^*}(\ovl{y})
=v \circ \ovl{u}(\ovl{y})
=v \circ \ovl{u(y)}
=v(\ovl{y})
=\varphi_y.
$$
Now, it suffices to use Theorem \ref{Th-Werner-bis} with $Y=X_*$ since the composition $v \circ \ovl{v^*} \co \ovl{Y^*} \to Y$ is equal to $j \co \ovl{X} \to X_*$, $\ovl{y} \mapsto \varphi_y$.
\end{proof}

Now, we give some geometric properties of these $\L^p$-spaces.

\begin{prop}
\label{Prop-smooth-bis}
Let $X$ be a $\JBW^*$-triple equipped with a functional $\varphi \in X_*$ with $\norm{\varphi}_{X_*}=1$ such that $s(\varphi)$ is complete. Suppose $1<p<\infty$. The Banach space $\L^p(X)$ is uniformly convex (hence reflexive) and uniformly smooth.
\end{prop}

\begin{proof}
By the reiteration theorem \cite[Theorem 4.6.1 p.~101]{BeL76}, the Banach space $\L^p(X)$, $1 < p < 2$ is a complex interpolation space between the predual $X_*$ and the space $\L^2(X)$. Since the latter is an Hilbert space by Proposition \ref{prop-L2-nonassociative}, it results from \cite{CwR82} that $\L^p(X)$ is uniformly convex. The same argument works for $2 < p < \infty$, since then the space $\L^p(X)$ is a complex interpolation space between $\ovl{X}$ and $\L^2(X)$. The last assertion is a consequence of \cite[Theorem 5.5.12 p.~500]{Meg98} which says that a normed space is uniformly convex if and only if its dual space is uniformly smooth.
\end{proof}

We finish this section with the following result which will be used in Example \ref{Ex-bicontractive}.

\begin{prop}
\label{prop-iso-Lp}
Let $X$ be a $\JBW^*$-triple equipped with a functional $\varphi \in X_*$ with $\norm{\varphi}_{X_*}=1$ such that $s(\varphi)$ is complete. Let $\alpha$ be an automorphism of $X$ satisfying $\varphi \circ \alpha = \varphi$. Then, for each $1 \leq p < \infty$, the map $\alpha$ induces a surjective isometry $\alpha_p \co \L^p(X) \to \L^p(X)$.
\end{prop}

\begin{proof}
Recall that a triple automorphism is an isometry, see e.g. \cite[Theorem 3.1.20 p.~183]{Chu12}. Hence $\norm{\alpha^{-1}(s(\varphi))}_X=\norm{s(\varphi)}_X=1$ and $\varphi(\alpha^{-1}(s(\varphi)))=\varphi(s(\varphi))=1$. Consequently $\varphi(\{x,y,\alpha^{-1}(s(\varphi))\}) 
\ov{\eqref{Def-scalar-varphi-bis}}{=} \langle x, y \rangle_{\cal{H}_\varphi}$ for any $x,y \in X$. Using the invariance $\varphi \circ \alpha = \varphi$, we see that for any $x,y \in X$
\begin{align*}
\MoveEqLeft
\varphi_{\alpha(y)}(\alpha(x))
\ov{\eqref{def-varphix}}{=} \varphi(\{\alpha(x),\alpha(y),s(\varphi)\}) 
=\varphi(\alpha(\{x,y,\alpha^{-1}(s(\varphi))\})) \\
&=\varphi(\{x,y,\alpha^{-1}(s(\varphi))\})
= \langle x, y \rangle_{\cal{H}_\varphi}
\ov{\eqref{Def-scalar-varphi-bis}}{=}\varphi(\{x,y,s(\varphi)\})
\ov{\eqref{def-varphix}}{=} \varphi_y(x). 
\end{align*}
We infer that the map $j(X) \to j(X)$, $\varphi_y \mapsto \varphi_{\alpha(y)}$, induces a surjective isometry $\alpha_1 \co \L^1(X) \to \L^1(X)$, which sends $j(X)$ into itself isometrically since $\alpha$ is an isometry. Consequently, we obtain the result by interpolation (here we use the fact that a linear contraction with contractive linear inverse is necessarily isometric). 
\end{proof}

\begin{remark} \normalfont
\label{remark-link}
If $X$ is a $\W^*$-$\TRO$, the link between these $\L^p$-spaces and the rectangular $\L^p$-spaces of Section \ref{Sec-complementation} is unclear.
\end{remark}

\section{Open questions}
\label{Sec-questions}
 
The following naive conjecture tries to clarify the open question of \cite[p.~99]{ArF78} on contractively complemented subspaces of noncommutative $\L^p$-spaces. Moreover, a clarification of Remark \ref{remark-link} is necessary in order to see if it is a good statement.

\begin{conj}
\label{conj1}
Suppose $1 < p < \infty$ with $p \not=2$. Let $Y$ be a Banach space. Then $Y$ is isometric to a contractively complemented subspace of a Haagerup noncommutative $\L^p$-space $\L^{p}(\cal{M},\psi)$ where $\cal{M}$ is a $\sigma$-finite von Neumann algebra equipped with a normal faithful state $\psi$ if and only if $Y$ is isometric to the $\L^p$-space $\L^p(X,\varphi)$ of a $\JW^*$-triple $X$ equipped with a linear functional $\varphi \in X_*$ with $\norm{\varphi}_{X_*}=1$ such that the support tripotent $s(\varphi)$ is complete.
\end{conj}

It is not clear if the following question has an affirmative answer. A suitable notion of trace exists on $\JBW^*$-algebras.

\begin{quest}
\label{conj2}
Does there exists a useful notion of trace on $\JBW^*$-triples which allows us to define $\L^p$-spaces associated to $\JBW^*$-triples ?
\end{quest}

We continue with the following precise question which asks if the category of $\L^p$-spaces of $\JBW^*$-triples and contractions is projectively stable \cite[p.~295]{NeR11}.

\begin{quest}
\label{quest-1}
Suppose $1 < p < \infty$ with $p \not=2$. Is it true that each contractively complemented subspace of an $\L^p$-space of $\JBW^*$-triple is isometric to another $\L^p$-space of a $\JBW^*$-triple ?
\end{quest}

Recall that a contractive projection $P \co X \to X$ on a Banach space $X$ is said to be bicontractive if the projection $\Id-P$ is also contractive and that an isometry $T \co X \to X$ such that $T^2=\Id$ is called a symmetry (or involutive isometry). It is easy to see and well-known that every symmetry $T \co X \to X$ gives rise to a bicontractive projection $P=\frac{1}{2}(\I+T)$ with complement $\Id-P=\frac{1}{2}(\Id-T)$. Nevertheless, if $P \co X \to X$ is a bicontractive projection on $X$ then in general $T=2P-\Id$ need not be a symmetry. It is a classical topic to determine if a Banach space (or a class of Banach spaces) has the property that each bicontractive projection $P \co X \to X$ has the form $P=\frac{1}{2}(\I+T)$ for some symmetry $T \co X \to X$. This problem is explicitly written in \cite[Problem 2 p.~261]{Rus94}.

It is known that $\L^p$-spaces of measures spaces and $\JB^*$-triples have this property, see \cite{ByS72}, \cite{BeL77} and \cite{FrB87}. Moreover, if $1 \leq p \leq \infty$ it is proved in \cite[Theorem 9.1]{ArF92} that each Schatten space $S^p$ has also this property. Consequently, the following problem is natural.

\begin{prob}
\label{prob-1}
Suppose $1 < p < \infty$ with $p \not=2$. Describe arbitrary bicontractive projections acting on noncommutative $\L^p$-spaces.
\end{prob}

It is easy to give examples of bicontractive projections acting on noncommutative $\L^p$-spaces.

\begin{example} \normalfont
\label{Ex-bicontractive}
Let $\cal{M}$ be a von Neumann algebra equipped with a normal faithful state $\varphi$. Consider a $*$-automorphism $\alpha$ of $\cal{M}$ satisfying $\alpha^2=\Id_\cal{M}$ and preserving the state, i.e. $\varphi \circ \alpha = \varphi$. Then $\frac{1}{2}(\Id+\alpha_p) \co \L^p(\cal{M}) \to  \L^p(\cal{M})$ is a bicontractive projection.
\end{example}

Now, we state a problem similar to Problem \ref{prob-1}.

\begin{prob}
\label{prob-2}
Suppose $1 < p < \infty$ with $p \not=2$. Describe the bicontractive projections acting on the $\L^p$-spaces $\L^p(X)$ of a $\JBW^*$-triple $X$ equipped with a linear functional $\varphi \in X_*$ with $\norm{\varphi}_{X_*}=1$ such that the support tripotent $s(\varphi)$ is complete.
\end{prob}

Similarly to Example \ref{Ex-bicontractive}, we can give examples of bicontractive projections on the $\L^p$-spaces of a $\JBW^*$-triple. 

\begin{example} \normalfont
\label{Ex-bicontractive}
Let $X$ be a $\JBW^*$-triple equipped with a linear functional $\varphi \in X_*$ with $\norm{\varphi}_{X_*}=1$ such that the support tripotent $s(\varphi)$ is complete. If $\alpha$ is an automorphism of $X$ satisfying $\alpha^2=\Id_X$ and $\varphi \circ \alpha = \varphi$ then with Proposition \ref{prop-iso-Lp}, we see that $\frac{1}{2}(\Id+\alpha_p) \co \L^p(X) \to  \L^p(X)$ is a bicontractive projection. Note that isometries of Cartan factors are described in \cite[Chapter 19]{Isi19}.
\end{example}

Let $\cal{M}$ be a von Neumann algebra. Recall that the noncommutative $\L^p$-space $\L^p(\cal{M},\varphi)$ is independent of the normal semifinite faithful weight $\varphi$ up to an isometric isomorphism, see \cite[p.~59]{Ter81} and \cite[Theorem 5.1]{Ray03}. So the next question is natural.
 
\begin{quest} 
\label{Open-quest}
Suppose $1 < p < \infty$. Let $X$ be a $\JBW^*$-triple equipped with a functional $\varphi \in X_*$ with $\norm{\varphi}_{X_*}=1$ such that $s(\varphi)$ is complete. Is it true that the Banach space $\L^p(X,\varphi)$ is independent from $\varphi$ ?
\end{quest}

We finish with the following question which generalizes the question of \cite[Remark 3.6]{ArK23}. A positive answer would allows us to improve Proposition \ref{prop-cd-proj-v1v2-proj}.

\begin{quest} 
\label{Open-quest}
Suppose $1 < p < \infty$. Let $\cal{M}$ and $\cal{N}$ be von Neumann algebras. Consider some $n \in \{1,2,\ldots,\infty\}$. If $T \co \L^p(\cal{M}) \to \L^p(\cal{M})$ is a contractively $n$-pseudo-decomposable map, does there exist some linear maps $v_1,v_2$ such that the map $\Phi$ of \eqref{Matrice-2-2-Phi} is $n$-positive and contractive ?
\end{quest}

\vspace{0.2cm}

\textbf{Acknowledgment}.
The author acknowledges support by the grant ANR-18-CE40-0021 (project HASCON) of the French National Research Agency ANR. The author would like to thank Ondrej Kalenda for a useful discussion on $\sigma$-finite $\JBW^*$-triples, Miguel Cabrera Garcia and {\'A}ngel Rodriguez Palacios for providing some useful information on $\JBW^*$-triples and finally Hermann Pfitzner for some corrections. Finally, the author would like to thank the referees for suggestions and corrections.


{\footnotesize

\vspace{0.2cm}

\noindent C\'edric Arhancet\\ 
\noindent 6 rue Didier Daurat, 81000 Albi, France\\
URL: \href{http://sites.google.com/site/cedricarhancet}{https://sites.google.com/site/cedricarhancet}\\
\url{cedric.arhancet@protonmail.com}\\

}

\end{document}